\documentclass[12pt]{iopart}

\usepackage{iopams,hyperref}
\usepackage{amsthm}
\usepackage{graphicx,subfigure,color}

\def\Id{{\rm Id}}
\def\op{{\rm op}}
\def\vec{{\rm vec}}
\def\eps{\varepsilon}
\def\N{\mathbb N}
\def\R{\mathbb R}

\begin{document}
\newtheorem{proposition}{Proposition}[section]
\newtheorem{definition}{Definition}[section]
\newtheorem{theorem}{Theorem}[section]
\newtheorem{lemma}{Lemma}[section]
\newtheorem{corollary}{Corollary}[section]
\newtheorem{remark}{Remark}[section]

\title[Corrections to Wigner type phase space methods]{Corrections to Wigner type phase space methods}

\author{Wolfgang Gaim$^1$ and Caroline Lasser$^2$}
\address{$^1$ Fachbereich Mathematik, Eberhard Karls Universit\"at T\"ubingen, 72076 T\"ubingen, Germany}
\address{$^2$ Zentrum Mathematik, Technische Universit\"at M\"unchen, 80290 M\"unchen, Germany}
\eads{\mailto{wolfgang.gaim@uni-tuebingen.de}, \mailto{classer@ma.tum.de}}

\begin{abstract}
Over decades, the time evolution of Wigner functions along classical Hamiltonian flows has been used for 
approximating key signatures of molecular quantum systems. Such approximations are for example the Wigner phase space method, the linearized semiclassical initial value representation, or the statistical quasiclassical method. 
The mathematical backbone of these approximations is Egorov's theorem. In this paper, we reformulate 
the well-known second order correction to Egorov's theorem as a system of ordinary differential equations 
and derive an algorithm with improved asymptotic accuracy for the computation of expectation values. For models with easily evaluated higher order derivatives of the classical Hamiltonian, the new algorithm's corrections are computationally less expensive than the leading order Wigner method. Numerical test calculations for a two-dimensional torsional system confirm the theoretical accuracy and efficiency of the new method.
\end{abstract}


\submitto{\NL}

\maketitle

\section{Introduction}
Molecular quantum systems are described by an unbounded self-adjoint operator, the Schr\"odinger operator 
\begin{equation*}
H = -\case{\eps^2}{2}\Delta + V,
\end{equation*}
acting on the Hilbert space of complex-valued square-integrable functions $L^2(\R^d)$. 
Here, $\eps>0$ is a small positive parameter related to the inverse square root of the average nuclear mass, and $V:\R^d\to\R$ is the nuclear potential function resulting from the Born--Oppenheimer approximation, see \cite{ST01}. 
The relevant time scales of nuclear quantum motion are of the order~$1/\eps$. The time evolution of an 
initial wave function $\psi_0\in L^2(\R^d)$ is governed by the time-dependent linear Schr\"odinger equation 
\begin{equation}\label{schroedinger}
\rmi \eps\partial_t \psi_t = H\psi_t
\end{equation}
with the appropriate $\eps$-scaling of the time-derivative or -- equivalently -- by the action of the unitary evolution operator $\e^{-iHt/\eps}$, since 
\begin{equation*}
\psi_t=\rme^{-\rmi Ht/\eps}\psi_0\qquad (t\in\R).
\end{equation*} 

Even though the Schr\"odinger equation \eref{schroedinger} is a {\em linear} partial differential equation, the numerical simulation of physical quantities derived from the wave function $\psi_t$ is notoriously difficult for two reasons: 
The dimension $d\gg1$ of the nuclear configuration space is large. If a molecule consists of $n$ nuclei, then $d=3n$. Just for a single water molecule, for example, we have $d = 3\cdot 3 = 9$. 
Moreover, nuclear quantum motion is highly oscillatory. Solutions typically oscillate with frequencies of the order $1/\eps$ in time and space, while $\eps$ ranges between $0.001$ and $0.1$, depending on the molecular system under consideration. For the hydrogen molecule H$_2$, for example, one has $\eps\approx0.0233$, while $\eps\approx0.0035$ for iodine monobromide IBr.

As an answer to these challenges, chemical physicists have developed approximations involving the {\em nonlinear} time evolution of classical mechanics. One uses the Hamiltonian function 
\begin{equation}\label{hamiltonian}
h:\R^d\times\R^d\to\R,\quad h(q,p) = \case{1}{2}|p|^2 + V(q),
\end{equation}
the associated Hamiltonian system 
\begin{equation}\label{hamilton}
\dot q_t = \partial_p h(q_t,p_t),\qquad \dot p_t = -\partial_q h(q_t,p_t),
\end{equation}
and the corresponding Hamiltonian flow 
\begin{equation*}
\Phi^t:\R^{2d}\to\R^{2d}
\end{equation*} 
as the classical counterparts to the Schr\"odinger operator $H$, the time-dependent Schr\"odinger equation~\eref{schroedinger}, and the unitary evolution operator 
$\rme^{-\rmi H t/\eps}$. 
This quantum-classical correspondence is most elegantly elaborated by using Wigner functions and Weyl quantization. 

The Wigner function $W_{\psi}: \R^{2d}\to\R$ of a square integrable function $\psi\in L^2(\R^d)$ is a real-valued, square integrable, continuous function on phase space $\R^{2d}$ obtained by an inverse Fourier transform of the autocorrelation function of $\psi$, see the short summary in~\ref{app:wigner}.  The Wigner function can be thought of as a probability density on phase space $\R^{2d}$, though in general it might attain negative values. It has been introduced by E.~Wigner in \cite{Wi32} when developing the thermodynamics of quantum mechanical systems. Crucial properties of the Wigner function are the orthogonality relation
\begin{equation*}
|\langle\phi,\psi\rangle|^2 = (2\pi\eps)^d \int_{\R^{2d}} W_{\phi}(z) W_{\psi}(z) \,\rmd z\qquad \left(\phi,\psi\in L^2(\R^d)\right)
\end{equation*}
and the asymptotically classical time evolution 
\begin{equation}\label{wigner}
W_{\psi_t} = W_{\psi_0}\circ\Phi^{-t} + O(\eps^2),\qquad \eps\to0,
\end{equation}
for the solution $\psi_t$ of the Schr\"odinger equation \eref{schroedinger}. The Wigner phase space method of E.~Heller \cite{He76,BH81} and the statistical quasiclassical method of H.~Lee and M.~Scully \cite{LS80}, for example, use these properties for computing the transition probabilities from a given state $\phi$ to $\psi_t$ according to
\begin{equation*}
|\langle\phi,\psi_t\rangle|^2  
\approx  (2\pi\eps)^d \int_{\R^{2d}} W_{\phi}(\Phi^t(z)) W_{\psi_0}(z) \,\rmd z.
\end{equation*}

Viewing the Wigner function $W_\psi$ of a square integrable function $\psi$ as a tempered distribution, that is, as a continuous linear mapping from the Schwartz functions $a:\R^{2d}\to\R$ to the real numbers, one arrives at the identity
\begin{equation}\label{eq:exp}
\int_{\R^{2d}} a(z) W_\psi(z) \,\rmd z = \langle \psi,\op(a)\psi\rangle,
\end{equation}
where $\op(a)$ denotes the bounded linear operator on $L^2(\R^d)$ obtained by the Weyl quantization of the function $a$. The Weyl operator $\op(a)$ is a pseudo-differential operator, a generalized partial differential operator, which treats differentiation and multiplication by functions on the same footing, see the appendix. With an appropriate handling of operator domains, Weyl quantization also applies for unbounded linear operators, and the Schr\"odinger operator $H=\op(h)$, for example, is the Weyl quantized classical Hamiltonian function $h$ defined in \eref{hamiltonian}. Often, the Weyl operator $\op(a)$ is also called the quantum observable assocaited with the classical observable $a$. Besides the convenient relation \eref{eq:exp}, Weyl quantization enjoys the beautiful property that the trace of the product of Weyl operators can be expressed as the integral
\begin{equation*}
\tr\!\left(\op(a)\op(b)\right) = (2\pi\eps)^{-d} \int_{\R^{2d}} a(z) b(z) \,\rmd z,
\end{equation*}
provided that the classical observables $a$ and $b$ satisfy appropriate regularity and growth conditions, see e.g. \cite[Proposition~9.2]{DS99} or \cite[Proposition~284]{Go11}.

From the Weyl point of view, the asymptotic time evolution of the Wigner function~\eref{wigner} can equivalently be formulated as 
\begin{equation}\label{egorov}
\rme^{\rmi Ht/\eps} \op(a) \rme^{-\rmi Ht/\eps} = \op(a\circ\Phi^t) + O(\eps^2),\qquad \eps\to0.
\end{equation}
This quantum-classical approximation of time evolved quantum observables is known as Egorov's theorem and has first been formulated in \cite{Eg69} within the H\"ormander theory of pseudo-differential operators, see also \cite{Ro87,BR02,Zw12} for refined error estimates in the context of semiclassical microlocal analysis. The computational power of classically propagating quantum observables has been recognized by W.~Miller and H.~Wang \cite{Mi74,WSM98}, who approximate the time-dependent correlation function of $\op(a)$ and $\op(b)$ according to
\begin{equation*}
\tr\!\left(\rme^{\rmi Ht/\eps}\op(a)\rme^{-\rmi Ht/\eps}\op(b)\right) \approx 
(2\pi\eps)^{-d} \int_{\R^{2d}} a(\Phi^t(z)) b(z)\,\rmd z.
\end{equation*}
In the chemical physics literature this approximation is referred to as the linearized semiclassical initial value representation (LSC-IVR), and it is derived from Fourier integral operator representations of the unitary evolution operator $\rme^{-\rmi Ht/\eps}$, see also \cite{TW04} for a recent review. It seems that the link of LSC-IVR to Egorov's theorem has not been noticed so far. 

In a study of the one-dimensional Morse oscillator, H.~Lee and M.~Scully \cite{LS82} have proposed to improve the second order approximation~\eref{wigner} for the Wigner function of $\psi_t$ by $W_{\psi_t} \approx W$ with
$$
\partial_t W = -p \;\partial_q W  + V_{\rm eff}'\;\partial_p W,
$$ 
where the derivative of the effective potential $V_{\rm eff}$ is given by 
$$
V_{\rm eff}' = V' - \frac{\eps^2}{24} \gamma\, V'''.
$$ 
For the correcting factor  $\gamma=\gamma(q,p,t)$ they mention the three options
$$
\gamma_1 = \partial^3_p(W_{\psi_0}\circ\Phi^{-t})/\partial_p W,\quad
\gamma_2 = \partial^3_p(W_{\psi_0}\circ\Phi^{-t})/\partial_p (W_{\psi_0}\circ\Phi^{-t}),\quad
\gamma_3 = \partial^3_p W/\partial_p W,
$$
and they use the term Wigner trajectories for the solutions of the ordinary differential equation
$\dot q_t = p_t$, $\dot p_t = -V_{\rm eff}'(q_t,p_t,t)$,  
see also \cite{L92, L95}. Later, A.~Donoso and C.~Martens~\cite{DM01} have defined their entangled classical trajectories by a fourth variant $\gamma_4 = \partial_p^2 W/W$. Trajectories generated by $\gamma_3$ and $\gamma_4$ and generalisations thereof have also been incorporated for the numerical simulation of time-dependent correlation functions of one-dimensional systems by J.~Liu and W.~Willer \cite{LM07}.  

Our aim here is also a higher order approximation of the dynamics. However, we will work with the rigorous mathematical framework of Egorov's theorem~\eref{egorov} and construct phase space trajectories without any entanglement for arbitrary dimensions~$d\ge1$.
The key element for proving Egorov's theorem is an asympotic expansion of the commutator 
\begin{equation}\label{commutator}
\frac{\rmi}{\eps}\left[\op(h),\op(a)\right] \;\sim\; \sum_{k\in 2\N} \left(\frac{\eps}{2\rmi}\right)^{k} \op(\{h,a\}_{k+1}).
\end{equation}
where $\{h,a\}_{k+1}$ denotes a generalization of the usual Poisson bracket involving derivatives of the functions $h$ and $a$ up to order $k+1$, see \Sref{corrections}. In the context of response theory, M.~Kryvohuz and J.~Cao \cite{KC05} use this expansion up to the fourth term for a systematic improvement of linear response computations over long times. The commutator expansion \eref{commutator} also reveals, that for Hamiltonians $h$, wich are polynomials of degree less or equal than two, the remainder of Egorov's theorem vanishes.
This exact Egorov result is utilized by H.~Waalkens, R.~Schubert and S.~Wiggins for their quantum normal form algorithm in dynamical transition state theory~\cite{WSW08}. 

For general Hamiltonian functions $h$, the expansion \eref{commutator} implies a higher order version of Egorov's theorem, 
\begin{equation}\label{egorov_expansion}
\rme^{\rmi Ht/\eps} \op(a) \rme^{-\rmi Ht/\eps}\sim\sum_{k\in 2\N}\eps^k\, \op(a_k(t))
\end{equation}
with leading order term $a_0(t) = a\circ\Phi^t$ and corrections 
\begin{equation*}
a_k(t) = \sum_{l\in\{0,2,\ldots,k-2\}} \left(\case{\rmi}{2}\right)^{k-l}\;\int_0^t \{h,a_l(\tau)\}_{k+1-l}\circ\Phi^{t-\tau} \,\rmd\tau.
\end{equation*}
It is remarkable that the corrections are only built from derivatives of the observable~$a$, the Hamiltonian function~$h$ and the flow $\Phi^t$, which has also been emphasized by M.~Pulvirenti in \cite{Pu06}, when deriving higher order estimates for the Wigner function of the wave function $\psi_t$. 
 
Since the higher order Egorov expansion~\eref{egorov_expansion} is built of even powers of the parameter~$\eps$, the first  correction 
\begin{equation}\label{def:a2}
a_2(t) = - \case{1}{4} \int_0^t \{h,a\circ\Phi^\tau\}_3\circ\Phi^{t-\tau} \,\rmd\tau
\end{equation}
provides the fourth order estimate
\begin{equation*}
\rme^{\rmi Ht/\eps} \op(a) \rme^{-\rmi Ht/\eps} = \op(a\circ\Phi^t + \eps^2 a_2(t)) + O(\eps^4).
\end{equation*}  
The aim of this paper is the exemplary analysis of this correction and its discretization for the numerical computation of expectation values. In a first step, we reformulate it as 
\begin{equation*}
a_2(t) = \Lambda_1^t(Da\circ\Phi^t) + \Lambda_2^t(D^2a\circ\Phi^t) + \Lambda_3^t(D^3a\circ\Phi^t),
\end{equation*}
where 
\begin{equation*}
\Lambda_k^t: \R^{2d\times\cdots\times 2d}\to\R \qquad (k=1,2,3),
\end{equation*} 
is an explicitly defined linear mapping from the space of $k$-tensors to the real numbers, which depends on derivatives of the Hamiltonian function $h$ and the flow $\Phi^t$, but not on the observable $a$. In spirit, this formulation of $a_2(t)$ is close to the Wigner trajectories generated by the first correction factor $\gamma_1$. In the next step, we derive a first order ordinary differential system for the components of $\Lambda_1^t$, $\Lambda_2^t$, and $\Lambda_3^t$. The vectorization of this equation results in
\begin{equation}\label{Lambda}
\frac{d}{dt}{\Lambda^t} = N^t \Lambda^t + C^t,
\end{equation}
where the building blocks of the matrix $N^t$ and the vector $C^t$ are components of the tensors $D^{2} h$,  $D^{3} h$, and $D^{4} h$ evaluated along the flow $\Phi^t$. In the final step we derive a fourth order splitting scheme for the discretization of the ordinary differential equation~\eref{Lambda}, which is then applied for numerical test calculations. 

This paper is organized as follows. \Sref{corrections} develops Egorov's theorem to the next order with respect to the parameter $\eps$ and formulates this correction as a first order ordinary differential equation. \Sref{discretization} discusses a discretization of this corrected approximation for the computation of expectation values. \Sref{numerical} provides numerical experiments for a two-dimensional torsional quantum system, which confirm the theoretical considerations. The appendix summarizes basic properties of Wigner functions and Weyl operators. 

\section{Higher order Corrections}\label{corrections}

Let $h:\R^{2d}\to\R$ be a smooth function of subquadratic growth. That is, for all $\gamma\in\N^{2d}$ with $|\gamma|\ge2$ there exists $C_\gamma>0$ with 
\begin{equation*}
\|D^\gamma h\|\le C_\gamma.
\end{equation*} 
Let $a:\R^{2d}\to\R$ be a Schwartz function. Then, the following formal considerations can be turned into a proof according to \cite[Th\'{e}or\`{e}me IV.10]{Ro87} or \cite[Theorem~1.2]{BR02}. 

We look for an approximate classical observable
$a_{\rm appr}(t):\R^{2d}\to\R$ such that 
\begin{equation*}
\rme^{\rmi Ht/\eps} \op(a) \rme^{-\rmi Ht/\eps}\approx \op(a_{\rm appr}(t)).
\end{equation*}
We require $a_{\rm appr}(0)=a$ at time $t=0$, rewrite the difference according to
\begin{eqnarray*}
\lefteqn{\rme^{\rmi Ht/\eps} \op(a) \rme^{-\rmi Ht/\eps} - \op(a_{\rm appr}(t))
= \int_0^t \frac{\rmd}{\rmd\tau} \left( \rme^{\rmi H\tau/\eps} \op(a_{\rm appr}(t-\tau)) \rme^{-\rmi H\tau/\eps}\right) \rmd\tau}\\
&=& \int_0^t \rme^{\rmi H\tau/\eps} \left( \frac{\rmi}{\eps}[\op(h),\op(a_{\rm appr}(t-\tau))] - 
\frac{\rmd}{\rmd t}\,\op(a_{\rm appr}(t-\tau)) \right) \rme^{-\rmi H\tau/\eps}\rmd\tau,
\end{eqnarray*}
and observe that the commutator 
\begin{equation*}
[\op(h),\op(a_{\rm appr}(t))] = \op(h)\op(a_{\rm appr}(t))-\op(a_{\rm appr}(t))\op(h)
\end{equation*} 
plays a crucial rule. Multiplying this commutator with $-\rmi \eps$, we obtain an asymptotic expansion in even powers of $\eps$, 
\begin{equation}\label{eq:ascomm}
\frac{\rmi}{\eps}[\op(h),\op(a_{\rm appr}(t))] \sim 
\sum_{k\in2\N} \left(\frac{\eps}{2i}\right)^k \op(\{h,a_{\rm appr}(t)\}_{k+1}),
\end{equation}
where
\begin{equation*}
		\{f,g\}_k = 
			\sum_{|\alpha+\beta|=k} \frac{(-1)^{|\beta|}}{\alpha! \beta!} \partial_q^\alpha \partial_p^\beta g\, \partial_q^\beta \partial_p^\alpha f. 
\end{equation*}
denotes the $k$th generalized Poisson bracket of two smooth functions $f,g:\R^{2d}\to\R$, see \ref{app:wigner}. Since the first generalized Poisson bracket coincides with the usual Poission bracket, $\{f,g\}_1 = \partial_p f\partial_q g - \partial_q f\partial_p g$, we have
\begin{equation*}
\frac{\rmi}{\eps}[\op(h),\op(a_{\rm appr}(t))] = \op(\{h,\op(a_{\rm appr}(t)\}) + O(\eps^2).
\end{equation*} 
Let $\Phi^t:\R^{2d}\to\R^{2d}$ be the flow associated with the classical Hamilton function $h$, and set 
$a_{\rm appr}(t)=a\circ\Phi^t$. Then, 
\begin{equation*}
\case{\rmd}{\rmd t}\, a_{\rm appr}(t) = \{ h,a_{\rm appr}(t)\},
\end{equation*} 
and we obtain Egorov's theorem,  
\begin{equation*}
\rme^{\rmi Ht/\eps} \op(a) \rme^{-\rmi Ht/\eps} = \op(a\circ\Phi^t) + O(\eps^2).
\end{equation*}

\subsection{The first correction}
The asymptotic commutator expansion \eref{eq:ascomm} allows to systematically derive higher order corrections to Egorov's theorem. For constructing the first correction, we set $a_{\rm appr}(t) = a_0(t) + \eps^2 a_2(t)$. The computation 
\begin{equation*}
\frac{\rmi}{\eps}[\op(h),\op(a_{\rm appr}(t))] = \op(\{h,a_{\rm appr}(t)\}) - \frac{\eps^2}{4}\, \op(\{h,a_0(t)\}_3)+ O(\eps^4)
\end{equation*}
suggests to choose
\begin{equation*}
a_0(t) = a\circ\Phi^t,\qquad a_2(t) = -\case14 \int_0^t \{h,a_0(\tau)\}_{3}\circ\Phi^{t-\tau} \rmd\tau.
\end{equation*} 
Indeed, we compute the time derivative,
\begin{eqnarray*}
		\frac{\rmd}{\rmd t} a_2(t) &=&   -\case14 \left( \{ h,a_{0}(t)\}_{3} 
			+	\int_0^t \frac{\rmd}{\rmd t}\left(\{h, a_{0}(\tau)\}_{3}  \circ \Phi^{t-\tau}\right)\rmd\tau  \right)
			\\&=& -\case14 \left( \{ h,a_{0}(t)\}_{3} 
			+	\int_0^t \{h,\{h, a_{0}(\tau)\}_{3}\}   \circ \Phi^{t-\tau}\rmd\tau  \right)
			\\&=&   -\case14 \left( \{ h,a_{0}(t)\}_{3} 
			+	\{h,\int_0^t\{h, a_{0}(\tau)\}_{3}   \circ \Phi^{t-\tau}\rmd\tau\}  \right),
\end{eqnarray*}
where the last equation uses that the classical flow as a symplectic transformation of phase space preserves the Poisson bracket.
We therefore obtain
\begin{equation*}
\frac{\rmd}{\rmd t} a_{\rm appr}(t) = 
\{h,a_{\rm appr}(t)\} - \frac{\eps^2}{4} \{h,a_0(t)\}_{3}.
\end{equation*}
Consequently,
\begin{equation*}
\rme^{\rmi Ht/\eps} \op(a) \rme^{-\rmi Ht/\eps} = \op(a_{\rm appr}(t)) + O(\eps^4).
\end{equation*} 

\begin{remark}\label{rem:inv} If $a:\R^{2d}\to\R$ is an observable invariant along the Hamiltonian flow, that is, $a\circ\Phi^t=a$ for all $t\in\R$, and additionally $\{h,a\}_3=0$, then the corrected approximation is also invariant, 
\begin{equation*}
a_{\rm appr}(t) = a\circ\Phi^t - \case{\eps^2}{4}\int_0^t \{h,a\circ\Phi^\tau\}_3\circ\Phi^{t-\tau} \rmd \tau = a
\end{equation*}
for all $t\in\R$. In particular, mass ($a=1$) and energy ($a=h$) are conserved.  
\end{remark}

\subsection{Ordinary Differential Equations for the Correction}

Our next aim is to reformulate the correction as
\begin{equation*}
a_2(t) = -\case{1}{4} \left( (D^3a\circ\Phi^t)_{ijk} \Lambda^t_{kji} + 3(D^2a\circ\Phi^t)_{ij}\Gamma_{ji}^t + (Da\circ\Phi^t)_i \Xi^t_i \right),
\end{equation*}
where the components of the time-dependent tensors $\Lambda^t$, $\Gamma^t$, and $\Xi^t$ satisfy a first order system of coupled ordinary differential equations, that is independent of the observable~$a$ and can be efficiently solved alongside the Hamiltonian flow $\Phi^t$. Here and in the following, we use Einstein's summation convention for notational brevity. 

Let 
\begin{equation*}
J = \left(\begin{array}{cc}0 & \Id\\ -\Id & 0\end{array}\right)\in\R^{2d\times 2d}.
\end{equation*}
We observe that we can write
\begin{eqnarray*}
		a_2(t) &=& -\case14 \int_0^t \{h,a\circ\Phi^\tau\}_{3}\circ\Phi^{t-\tau} \rmd\tau\\
		&=& -\case{1}{4}\int_0^t \left( \left(D^3(a\circ \Phi^\tau)\right)_{lmn} 
				 \left(J\widetilde{D^3h}\right)_{nml}	\right) \circ	\Phi^{t-\tau} \rmd\tau,
\end{eqnarray*}				 
where the 3-tensor $\left(J\widetilde{D^3h}\right)$ is defined by 
\begin{eqnarray*} 
		\eqalign{
		\left(\widetilde{D^3h}\right)_{ijk}= 
			\cases{
  			\case{1}{6} (D^3h)_{ijk},  & $i=j=k$\\
  			\case{1}{2} (D^3h)_{ijk}, & $i=j\neq k$ or $i=k\neq j$ or $k=j\neq i$ \\
  			(D^3h)_{ijk} , & else  		
			}}
\end{eqnarray*}
and 
\begin{equation}\label{def_JD3tilde} 
\left(J\widetilde{D^3h}\right)_{ijk} = J_{il} J_{jm} J_{kn} \left(\widetilde{D^3h}\right)_{lmn}.
\end{equation}
A closer look at this 3-tensor reveals that it is symmetric:

\begin{lemma}
\label{lem_symm_JD3h}
	Let $k\in\N$ and $A=(a)_{i_1 \cdots i_k} \in \mathbb{R}^{2d \times \cdots \times 2d}$ be a $k$-tensor.  Then, 
	\begin{equation*}
	J_{i_1 l_1} \cdots J_{i_k l_k} A_{l_1 \cdots l_k}\;\mbox{is symmetric}\;\Leftrightarrow 
	A \;\mbox{is symmetric}.
	\end{equation*}
In particular, the 3-tensor $\left(J\widetilde{D^3h}\right)$ defined in \eref{def_JD3tilde} is symmetric. 	 
\end{lemma}

\begin{proof}
	For every $m\in \{1,\ldots,k\}$
	\begin{eqnarray*}
		J_{i_m l} A_{i_1 \cdots i_{m-1} l i_{m+1} \ldots i_k} = 
			\cases{ 
				A_{i_1 \cdots i_{m-1} (i_m+d)  i_{m+1} \ldots i_k}, &$i_m \leq d$ \\
				-A_{i_1 \cdots i_{m-1} (i_m-d)  i_{m+1} \ldots i_k}, &else
			}.
	\end{eqnarray*}
Let $1\le j\le k$. For $i_j\le d$, we set $\alpha_j=0$, $\beta_j=i_j+d$ and otherwise $\alpha_j=1$, $\beta_j=i_j-d$. Then,   
	\begin{eqnarray*}
		J_{i_1 l_1} \cdots J_{i_k l_k} A_{l_1 \cdots l_k} = (-1)^{\alpha_1} \cdots (-1)^{\alpha_k}  A_{\beta_1 \ldots \beta_k}.	
	\end{eqnarray*}
	Hence, a permutation of $(i_1,\ldots,i_k)$ results in a permutation of $(\alpha_1,\ldots, \alpha_k)$ and $(\beta_1,\ldots, \beta_k)$. 
	Thus, $J_{i_1 l_1} \cdots J_{i_k l_k} A_{l_1 \cdots l_k}$ is symmetric if and only if $A$ is symmetric. 
\end{proof}

We compute the third derivative of $a\circ\Phi^\tau$ by the chain rule, where we denote the derivatives of $\Phi^\tau$ by
$\partial_{i_k}\cdots\partial_{i_1} \Phi^\tau_j=(D^k\Phi^\tau)_{ji_1 \cdots i_k}$,
\begin{eqnarray*}
		\fl \left(D^3(a\circ \Phi^\tau)\right)_{lmn} =& (D^3a \circ \Phi^\tau)_{ijk} (D\Phi^\tau)_{il}  (D\Phi^\tau)_{jm}  (D\Phi^\tau)_{kn} 
			\\&+ (D^2a \circ \Phi^\tau)_{ij} (D^2\Phi^\tau)_{ilm}  (D\Phi^\tau)_{jn} 
				 +	(D^2a)_{ij} (D^2\Phi^\tau)_{iln}  (D\Phi^\tau)_{jm} 
		  \\&+ (D^2a \circ \Phi^\tau)_{ij} (D^2\Phi^\tau)_{imn}  (D\Phi^\tau)_{jl} 
			   + (Da \circ \Phi^\tau)_{i} (D^3\Phi^\tau)_{ilmn}. 
	\end{eqnarray*}
Since the 3-tensor $\left(J\widetilde{D^3h}\right)$ is symmetric and $\Phi^\tau \circ \Phi^{t-\tau}= \Phi^t$, we obtain
the reformulation
	\begin{eqnarray*}
		a_2(t) =-\case{1}{4} \left( \left( D^3a \circ \Phi^t \right)_{ijk} \Lambda^t_{kji} + 
			3 \left( D^2a \circ \Phi^t \right)_{ij}  \Gamma^t_{ji} + \left( Da \circ \Phi^t \right)_{i} \Xi^t_i \right)
	\end{eqnarray*}			
	with
	\begin{equation}\label{eq:deften}
	\eqalign{		 
		\Lambda^t_{ijk} &= \int_{0}^t \left[ \left( D\Phi^\tau \right)_{il}  \left( D\Phi^\tau \right)_{jm}  
																			\left(D\Phi^\tau \right)_{kn} 
					\left(J\widetilde{D^3h}\right)_{nml} \right] \circ \Phi^{t-\tau} \rmd\tau , 
		\\ \Gamma^t_{ij} &= \int_{0}^t \left[ \left( D^2\Phi^\tau \right)_{ikl} \left( D\Phi^\tau \right)_{jm} 
					\left(J\widetilde{D^3h}\right)_{mlk} \right] \circ \Phi^{t-\tau} \rmd\tau ,
		\\ \Xi^t_{i} &= \int_{0}^t \left[ \left( D^3\Phi^\tau \right)_{ijkl} 
					\left(J\widetilde{D^3h}\right)_{lkj} \right] \circ \Phi^{t-\tau} \rmd\tau. }
	\end{equation}

Next, we compute the time derivatives of integrals, which are of the form observed in the defining equations of the time-dependent tensors $\Lambda^t$, $\Gamma^t$, and $\Xi^t$.

\begin{lemma}
\label{Lem_main_pde}
	Let $b:\R^{2d}\to\R$ and $f:\mathbb{R} \times \mathbb{R}^{2d} \rightarrow \mathbb{R}$ smooth functions. Then, 
	\begin{equation*}
		\fl\frac{\rmd}{\rmd t} \int_0^t \left( b f(\tau) \right) \circ \Phi^{t-\tau} \rmd\tau =
		\int_0^t \left( b \,\frac{\rmd}{\rmd\tau} f(\tau) \right) \circ \Phi^{t-\tau} \rmd\tau
		+  \left( bf(0) \right) \circ \Phi^t.
	\end{equation*}
\end{lemma}

\begin{proof}
We start with
\begin{eqnarray*}
\fl\frac{\rmd}{\rmd t} \int_0^t \left( b f(\tau) \right) \circ \Phi^{t-\tau} \rmd\tau 
		= b f(t) 
			+ 	\int_0^t b \circ \Phi^{t-\tau} 
			(Df(\tau))^T \circ \Phi^{t-\tau} \frac{\rmd}{\rmd t} \Phi^{t-\tau} \rmd\tau
			\\ 
			+		\int_0^t (Db)^T \circ \Phi^{t-\tau} \frac{\rmd}{\rmd t} \Phi^{t-\tau} 
				f(\tau)\circ \Phi^{t-\tau} \rmd\tau .
\end{eqnarray*}
For introducing the $\tau$-derivative in the integral, we compute
	\begin{eqnarray*}
\fl		 \frac{\rmd}{\rmd\tau} \left( b  f(\tau) \circ \Phi^{t-\tau} \right)   =
			b \circ \Phi^{t-\tau} \left((\frac{\rmd}{\rmd\tau} f(\tau))\circ \Phi^{t-\tau} 
			- (Df(\tau))^T \circ  \Phi^{t-\tau}	\frac{\rmd}{\rmd t} \Phi^{t-\tau} \right)
			\\
			- (Db)^T \circ \Phi^{t-\tau} \frac{\rmd}{\rmd t}  \Phi^{t-\tau} 
				f(\tau)\circ \Phi^{t-\tau}.
	\end{eqnarray*}
	Therefore, 
	\begin{eqnarray*}
\fl \frac{\rmd}{\rmd t} \int_0^t \left( b f(\tau) \right) \circ \Phi^{t-\tau} \rmd\tau\\		
		=	b f(t) + \int_0^t \left(b \frac{\rmd}{\rmd\tau} f(\tau)\right)\circ \Phi^{t-\tau} \rmd\tau
			- \int_0^t \frac{\rmd}{\rmd\tau} \left(\left( b
		 	f(\tau) \right)\circ \Phi^{t-\tau} \right)  \rmd\tau
		\\
		=
		 	\int_0^t \left(b  \frac{\rmd}{\rmd\tau} f(\tau) \right)\circ \Phi^{t-\tau} \rmd\tau
		 	+ \left(b f(0) \right) \circ \Phi^t  .
	\end{eqnarray*}
\end{proof}

Now we are ready to formulate and prove our first main result, the explicit system of ordinary differential equations describing the second order correction to Egorov's theorem. In M.~Zworski's recent monograph, the higher order 
terms are referred to ``as difficult to compute'', see \cite[\S11.1]{Zw12}. Our result shows nonetheless, that the computation is feasible. 

\begin{theorem}[Second Correction]\label{ode_thm_odes}
Let $h:\R^{2d}\to\R$ be a smooth function of subquadratic growth, $t\in\R$, and $\Phi^t:\R^{2d}\to\R^{2d}$ the Hamiltonian flow associated with $h$. Let $a:\R^{2d}\to\R$ be a Schwartz function. Then, there exists a constant $C=C(a,h,t)>0$ such that for all $\eps>0$
\begin{equation*}
\left\| \rme^{\rmi \op(h)t/\eps}\op(a)\rme^{-\rmi \op(h)t/\eps} - \op(a_0(t) + \eps^2 a_2(t))\right\| \le C\eps^4
\end{equation*}
where $a_0(t)=a\circ\Phi^t$ and 
\begin{equation*}
a_2(t) = -\case{1}{4} \left( \left( D^3a \circ \Phi^t \right)_{ijk} \Lambda^t_{kji} + 
			3 \left( D^2a \circ \Phi^t \right)_{ij}  \Gamma^t_{ji} + \left( Da \circ \Phi^t \right)_{i} \Xi^t_i \right)
\end{equation*} 
and the functions $\Lambda^t_{ijk}$, $\Gamma^t_{ij}$, $\Xi^t_{i}$, $i,j,k=1,\ldots,2d$, solve the ordinary differential system
	\begin{equation}\label{odes_first_approach}
	\eqalign{
		\frac{\rmd}{\rmd t}\Lambda^t_{ijk} = M^t_{il} \Lambda^t_{ljk} + M^t_{jl} \Lambda^t _{ilk} 
									+ M^t_{kl} \Lambda^t_{ijl} +  C^1_{ijk}(t), \\
		\frac{\rmd}{\rmd t}\Gamma^t_{ij} = \left( C^2_i(t) \right)_{kl} \Lambda^t_{lkj} + M^t_{il} \Gamma^t _{lj} 
									+ M^t_{jl} \Gamma^t_{il}, \\
		\frac{\rmd}{\rmd t}\Xi^t_{i} = \left( C^3_i(t) \right)_{jkl} \Lambda^t_{lkj} 
									+  3\left( C^2_i(t) \right)_{jk} \Gamma^t _{kj}+ M^t_{il} \Xi^t_{l}, \\
		\Lambda^0_{ijk}=\Gamma^0_{ij}=\Xi^0_{i}=0 , 
	}
	\end{equation}
	with $M^t = J\cdot D^2h \circ \Phi^t$ and 
	\begin{eqnarray*}
	C^1_{ijk}(t)  = \left( J\widetilde{D^3h} \circ	\Phi^{t} \right)_{ijk},\qquad
	\left( C^2_i(t) \right)_{jk} 
	= \left( J\cdot D^3h \circ \Phi^t \right)_{ijk},
	\\ \left( C^3_i(t) \right)_{jkl} = \left( J\cdot D^4h \circ \Phi^t \right)_{ijkl},
	\end{eqnarray*}
	where for a matrix $A\in \mathbb{R}^{2d \times 2d}$ and $k$-tensor $B\in\mathbb{R}^{2d \times \cdots \times 2d}$, $A\cdot B\in\mathbb{R}^{2d \times \cdots \times 2d}$ is given by 
	$(A\cdot B)_{i_1 \cdots i_k}=A_{i_1 s} B_{s i_2 \cdots i_k}$ and $J\widetilde{D^3h}$ is defined by (\ref{def_JD3tilde}). 
\end{theorem}

\begin{proof}
	The initial values $\Lambda^0_{ijk}=\Gamma^0_{ij}=\Xi^0_{i}=0$ are clear, since $\Lambda^t_{ijk}$, $\Gamma^t_{ij}$ and $\Xi^t_{i}$ are defined in \eref{eq:deften} as an integral from $0$ to $t$, over a smooth integrand. Next, we study the time derivative of $\Lambda^t_{ijk}$. For this, we write the Hamiltonian system as
\begin{equation}\label{eq:jham}
\frac{\rmd}{\rmd t} \Phi^t = J\, D h \circ \Phi^t,
\end{equation}	
and apply the differential operator $D$ to obtain the Jacobi stability equation
\begin{equation*}
\frac{\rmd}{\rmd t} ( D\Phi^t )_{ij}=\left( J \cdot D^2h \circ \Phi^t \right)_{im} ( D\Phi^t )_{mj}. 
\end{equation*}  
Then, by Lemma~\ref{Lem_main_pde}, 
	\begin{eqnarray*}
		\fl\frac{\rmd}{\rmd t}\Lambda^t_{ijk} &=& \int_0^t \left[ \frac{\rmd}{\rmd\tau} \left( \left( D\Phi^\tau \right)_{il}  \left( D\Phi^\tau \right)_{jm}  
												\left(D\Phi^\tau \right)_{kn} \right) \left(J\widetilde{D^3h}\right)_{lmn} \right] \circ \Phi^{t-\tau} \rmd\tau
		\\\fl&&+  
			\left[  \left( D\Phi^0 \right)_{il}  \left( D\Phi^0 \right)_{jm}  
												\left(D\Phi^0 \right)_{kn} \left(J\widetilde{D^3h}\right)_{lmn} \right] \circ \Phi^t
		\\\fl&=& 
			\int_0^t \left[  \left( J\cdot D^2h \circ \Phi^\tau \right)_{i\nu} \left( D\Phi^\tau \right)_{\nu l}  \left( D\Phi^\tau \right)_{jm}  
												\left(D\Phi^\tau \right)_{kn}  \left(J\widetilde{D^3h}\right)_{lmn} \right] \circ \Phi^{t-\tau} \rmd\tau
		\\\fl&&+ 
			\int_0^t \left[ \left( D\Phi^\tau \right)_{il} \left( J\cdot D^2h \circ \Phi^\tau \right)_{j\nu} \left( D\Phi^\tau \right)_{\nu m}  
												\left(D\Phi^\tau \right)_{kn} \left(J\widetilde{D^3h}\right)_{lmn} \right] \circ \Phi^{t-\tau} \rmd\tau
		\\\fl&&+ 
			\int_0^t \left[  \left( D\Phi^\tau \right)_{il}  \left( D\Phi^\tau \right)_{jm}  
												\left( J\cdot D^2h \circ \Phi^\tau \right)_{k\nu} \left(D\Phi^\tau \right)_{\nu n} 
												\left(J\widetilde{D^3h}\right)_{lmn} \right] \circ \Phi^{t-\tau} \rmd\tau
		\\\fl&&+   
			\left(J\widetilde{D^3h}\right)_{ijk} \circ \Phi^t.
	\end{eqnarray*}
	 So we can finish the proof for $\frac{\rmd}{\rmd t}\Lambda^t$ by using that $\Phi^\tau \circ \Phi^{t-\tau} = \Phi^t$. The proofs for $\frac{\rmd}{\rmd t}\Gamma^t$ and $\frac{\rmd}{\rmd t}\Xi^t$ are analogous: We differentiate the Jacobi stability equation to obtain
	\begin{eqnarray*}		
		\fl\frac{\rmd}{\rmd t} ( D^2\Phi^t )_{ijk}&=&\left( J \cdot D^3h \circ \Phi^t \right)_{imn} ( D\Phi^t )_{nk} ( D\Phi^t )_{mj} +  
				\left( J \cdot D^2h \circ \Phi^t \right)_{im} ( D^2\Phi^t )_{mjk},
		\\ 
		\fl\frac{\rmd}{\rmd t} ( D^3\Phi^t )_{ijkl}&=&\left( J \cdot D^4h \circ \Phi^t \right)_{imn\mu} ( D\Phi^t )_{n k} ( D\Phi^t )_{m j} ( D\Phi^t )_{\mu l} 
					\\&&+\left( J \cdot D^3h \circ \Phi^t \right)_{imn} ( D\Phi^t )_{mj} ( D^2\Phi^t )_{nkl}
				 	\\&&+ \left( J \cdot D^3h \circ \Phi^t \right)_{imn} ( D\Phi^t )_{nk} ( D^2\Phi^t )_{mjl}
				 	\\&&+ \left( J \cdot D^3h \circ \Phi^t \right)_{imn} ( D\Phi^t )_{nl} ( D^2\Phi^t )_{mjk}
					+\left( J \cdot D^2h \circ \Phi^t \right)_{im} ( D^3\Phi^t )_{mjkl} .
	\end{eqnarray*}	 
Moreover, we use 	 	  
	\begin{eqnarray*}
		\fl \frac{\rmd}{\rmd\tau} \left(\left( D^2\Phi^\tau \right)_{ikl} \left( D\Phi^\tau \right)_{jm} \right) \left(J\widetilde{D^3h}\right)_{mlk} 
		\\=
			\left[ (J\cdot D^3h \circ \Phi^\tau)_{i\mu \nu} (D\Phi^\tau)_{\nu l} (D\Phi^\tau)_{\mu k} 
			\left( D\Phi^\tau \right)_{jm} 
		\right. \\ \left. 
			+(J \cdot D^2h \circ \Phi^\tau)_{i\mu} (D^2\Phi)_{\mu kl} \left( D\Phi^\tau \right)_{jm} 
			 \right. \\ \left. 
			 + \left( D^2\Phi^\tau \right)_{ikl} \left( J\cdot D^2h \circ \Phi^\tau \right)_{j\mu} \left(D\Phi^\tau \right)_{\mu m} 
			 	\right] \left(J\widetilde{D^3h}\right)_{mlk},
	\end{eqnarray*}
	\begin{eqnarray*}
			\fl \frac{\rmd}{\rmd\tau} \left( D^3\Phi^\tau \right)_{ijkl} \left(J\widetilde{D^3h}\right)_{lkj}=&
				\left[(J\cdot D^4h \circ \Phi ^ \tau)_{i \mu \nu \eta} (D \Phi ^ \tau)_{\nu k} (D\Phi^\tau)_{\mu j} (D\Phi^\tau)_{\eta l}
		\right. \\ &\left.+ 3 (J\cdot D^3h \circ \Phi ^ \tau)_{i \mu \nu} (D\Phi ^ \tau)_{\nu k} (D^2 \Phi ^ \tau)_{\mu j l} 
		\right. \\ &\left.+ 
			 	(J\cdot D^2h\circ \Phi^\tau)_{i \mu} (D^3\Phi^\tau)_{\mu j k l} \right] \left(J\widetilde{D^3h}\right)_{lkj}.
	\end{eqnarray*}
as well as $(J\cdot D^3h \circ \Phi ^ \tau)_{i \mu \nu}=(J\cdot D^3h \circ \Phi ^ \tau)_{i \nu \mu}$ and $D^2 \Phi^0 = D^3 \Phi^0 = 0$. 
\end{proof}


\subsection{Vectorization: General Hamiltonians}\label{chapter_vec}

For the numerical simulation of the ordinary differential system \eref{odes_first_approach}, we vectorize the tensors. Recall, that for matrices $A \in \mathbb{R}^{n\times n}$ and $B \in \mathbb{R}^{m \times m}$, the Kronecker product $A\otimes B\in\R^{mn\times mn}$ is defined as  
\begin{equation*}
A\otimes B 
		= 
			\left(\begin{array}{ccc}
				A_{11} B & \cdots & A_{1n} B \\
				\vdots & & \vdots \\
				A_{n1} B & \cdots & A_{nn} B
			\end{array}\right).
\end{equation*}
That is, 
\begin{equation*}
(A\otimes B)_{ij}= A_{i_1 j_1} B_{i_2 j_2},\quad\mbox{if}\; i=(i_1-1)m+i_2, j=(j_1-1)m+j_2.
\end{equation*}
The vectorization of a k-tensor $C=(C)_{i_1 \cdots i_k}\in \mathbb{R}^{n_1 \times \cdots \times n_k}$ is defined as
	\begin{eqnarray*}
		\vec(C)=\left(\begin{array}{ccccccc} c_{1 \cdots 1} & c_{1 \cdots 12} & \cdots & c_{1 \cdots 1n_k} & c_{1 \cdots 121} & \cdots & c_{n_1 \cdots n_k}
			\end{array}\right)^T\in\R^{n_1\cdot\ldots\cdot n_k}.
	\end{eqnarray*}
That is, $\vec(C)_i = c_{i_1 \cdots i_k}$ if 	
	\begin{eqnarray*}
		i = (i_1-1) \prod_{l=2}^k n_l + (i_2-1) \prod_{l=3}^{k} n_l + \cdots + (i_{k-1}-1) n_k + i_k .
	\end{eqnarray*}
The following observation allows the vectorization of the products occuring on the right hand side of our differential equation \eref{odes_first_approach}.  

\begin{lemma}
	\label{prop_vectorisation}
	Let $A\in \mathbb{R}^{m \times m}$ a matrix, $B \in \mathbb{R}^{m \times \cdots \times m}$ an n-tensor, and $k\in\{1,\ldots,m\}$.
	Define the n-tensor $C\in \mathbb{R}^{m \times \cdots \times m}$ as 
	\begin{equation*}C_{i_1 \cdots i_n}=A_{i_k l} B_{i_1 \cdots i_{k-1} l i_{k+1} \cdots i_n} .\end{equation*} Then,  
	\begin{eqnarray*}
		\vec(C) = ( \overbrace{\Id_m \otimes \cdots \otimes \Id_m}^{k-1\:\hbox{ times}} 
			\otimes A \otimes \overbrace{\Id_m \otimes \cdots \otimes \Id_m}^{n-k \:\hbox{ times}} ) \cdot \vec(B).
	\end{eqnarray*}
\end{lemma}

\begin{proof}
	We define the matrix $\widetilde{A}=( \overbrace{\Id_m \otimes \cdots \otimes \Id_m}^{k-1\:\hbox{ times}} 
			\otimes A \otimes \overbrace{\Id_m \otimes \cdots \otimes \Id_m}^{n-k \:\hbox{ times}} )$. 
	That is, 	
	\begin{eqnarray*}
		\widetilde{A}_{il}=
			\cases{
				A_{i_k l_k}, & if $i_\mu = l_\mu$ for all $\mu \in \{1,\cdots,k-1,k+1,\cdots,n\}$ 
				\\ 0, & else 
			}
	\end{eqnarray*}	
	for $i=(i_1-1)m^{n-1}+\cdots+(i_{n-1}-1)m + i_n$ and $l=(l_1-1)m^{n-1}+\cdots+(l_{n-1}-1)m + l_n$.		
  Then,  
	\begin{eqnarray*}
		\vec(C)_i = \widetilde{A}_{il} \cdot \vec(B)_l = A_{i_k l_k} B_{i_1 \cdots i_{k-1} l_k i_{k+1} \cdots i_n} = C_{i_1 \cdots i_n}
	\end{eqnarray*}
	for $i=(i_1-1)m^{n-1}+\cdots+(i_{n-1}-1)m + i_n$. 
\end{proof}

Now we reformulate the results of Theorem~\ref{ode_thm_odes} in vectorized form. 

\begin{corollary}[Vectorization]\label{thm_vec_odes}
Consider the time-dependent tensors $M^t$ and $C^j(t)$, $j=1,2,3$, together with the Hamiltonian flow $\Phi^t$ of 
Theorem~\ref{ode_thm_odes}. Let the functions $\Lambda^t_{ijk}$, $\Gamma^t_{ij}$ and $\Xi^t_{i}$ solve the ordinary differential system
	\begin{equation*}
	\eqalign{
		\frac{\rmd}{\rmd t}\Lambda^t_{ijk} = M^t_{il} \Lambda^t_{ljk} + M^t_{jl} \Lambda^t _{ilk} 
									+ M^t_{kl} \Lambda^t_{ijl} +  C^1_{ijk}(t),\\
		\frac{\rmd}{\rmd t}\Gamma^t_{ij} = \left( C^2_i(t) \right)_{kl} \Lambda^t_{lkj} + M^t_{il} \Gamma^t _{lj} 
									+ M^t_{jl} \Gamma^t_{il},\\
		\frac{\rmd}{\rmd t}\Xi^t_{i} = \left( C^3_i(t) \right)_{jkl} \Lambda^t_{lkj} 
									+  3\left( C^2_i(t) \right)_{jk} \Gamma^t _{kj}+ M^t_{il} \Xi^t_{l}.
	}
	\end{equation*}
Then, 
	\begin{equation}\label{vectorisation_odes_1}
	\eqalign{
		\fl\frac{\rmd}{\rmd t} \left(\begin{array}{c}
			\Phi^t \\ \vec(\Lambda^t) \\ \vec(\Gamma^t) \\ \Xi^t
		\end{array}\right)
		=
		\left(\begin{array}{cccc}
			\Id_{2d} & 0 & 0 & 0 \\ 0 & K^t & 0 & 0 \\ 0 & D^t & L^t & 0 
			\\ 0 & C^3_{\rm m}(t) & 3 C^2_{\rm m}(t) & M^t  
		\end{array}\right)
		\left(\begin{array}{c}
			J\cdot D h \circ \Phi^t \\ \vec(\Lambda^t) \\ \vec(\Gamma^t) \\ \Xi^t
		\end{array}\right)
		+
		\left(\begin{array}{c}
			0 \\ C^1_{\rm v}(t) \\ 0 \\ 0
		\end{array}\right)
	}
	\end{equation}	
	with $C^1_{\rm v}(t)=\vec(C^1(t))$ and 
	\begin{eqnarray*}
	C^j_{\rm m}(t) &=&
			\left(\begin{array}{ccc}
				\vec(C^j_1(t)) & \cdots & \vec(C^j_{2d}(t))
			\end{array}\right)^T,\qquad j=2,3,
	\end{eqnarray*}
			and			
	\begin{eqnarray*}
		K^t &=& M^t \otimes \Id_{2d} \otimes \Id_{2d} + \Id_{2d} \otimes M^t \otimes \Id_{2d} + \Id_{2d} \otimes \Id_{2d} \otimes M^t ,\\
		D^t &=& \Id_{2d} \otimes C^2_{\rm m}(t),\qquad 
		L^t = M^t \otimes \Id_{2d} + \Id_{2d} \otimes M^t.
	\end{eqnarray*}
\end{corollary}

\begin{proof}
Applying Lemma~\ref{prop_vectorisation} to the $\Lambda_{ijk}^t$ part of the ordinary differential equation, we obtain
	\begin{eqnarray*}
		\fl\frac{\rmd}{\rmd t} \overrightarrow{\Lambda}^t &=& \frac{\rmd}{\rmd t} \vec \left(\Lambda^t_{ijk} \right) 
			= \vec \left( M^t_{il} \Lambda^t_{ljk} + M^t_{jl} \Lambda^t _{ilk} + M^t_{kl} \Lambda^t_{ijl} +  C^1_{ijk}(t) \right) 
			\\
			\fl&=&	\left(	M^t \otimes\Id_{2d} \otimes \Id_{2d} + \Id_{2d} \otimes M^t \otimes \Id_{2d} + \Id_{2d} \otimes \Id_{2d} \otimes M^t \right)
			 		\vec(\Lambda^t) + C^1_{\rm v}(t)
			\\
			\fl&=& K^t \vec(\Lambda^t) + C^1_{\rm v}(t).
	\end{eqnarray*}
	The proofs for $\vec(\Gamma^t)$ and $\Xi^t$ are analogous.
\end{proof}

\subsection{Vectorization: Schr\"odinger Hamiltonians}
We now analyse the vectorized system of ordinary differential equations \eref{vectorisation_odes_1} for the special case, that the Hamilton function is given by 
\begin{equation*}h(q,p)=\case{1}{2}|p|^2+V(q).\end{equation*}  
Then, 
\begin{equation}\label{eq:MC}
\fl D^2h = \left( \begin{array}{cc} D^2V & 0 \\ 0 & \Id_d \end{array} \right),\quad
M^t = J\cdot D^2h\circ\Phi^t = \left( \begin{array}{cc} 0 & \Id_d \\ -D^2V \circ \Phi^t_q & 0 \end{array} \right).
\end{equation}
Moreover, 
\begin{eqnarray*}
\fl D^3h = \cases{ (D^3V)_{ijk}, & $i,j,k \leq d$ \\ 0, & else } , \qquad
		D^4h = \cases{ (D^4V)_{ijkl}, & $i,j,k,l \leq d$ \\ 0, & else }.
\end{eqnarray*}	
To exploit this zero pattern, we reorder our system of differential equations. We set
\begin{eqnarray*}
\fl\left( \Lambda_{1}^t \right)_{ijk} 	= \Lambda_{ijk}^t,\qquad  
&\left( \Lambda_{2,1}^t \right)_{ijk} 	= \Lambda_{(i+d)jk}^t,\qquad
&\left( \Lambda_{2,2}^t \right)_{ijk} 	= \Lambda_{i(j+d)k}^t,\\
\fl\left( \Lambda_{2,3}^t \right)_{ijk} = \Lambda_{ij(k+d)}^t, \qquad
&\left( \Lambda_{3,1}^t \right)_{ijk} 	= \Lambda_{i(j+d)(k+d)}^t,\qquad
&\left( \Lambda_{3,2}^t \right)_{ijk} 	= \Lambda_{(i+d)j(k+d)}^t,\\
\fl\left( \Lambda_{3,3}^t \right)_{ijk} 	= \Lambda_{(i+d)(j+d)k}^t, 
&\left( \Lambda_{4}^t \right)_{ijk}    = \Lambda_{(i+d)(j+d)(k+d)}^t,\\
\fl\left( \Gamma_{1}^t \right)_{ij} 			= \Gamma_{ij}^t,\qquad
&\left( \Gamma_{2,1}^t \right)_{ij} 		= \Gamma_{(i+d)j}^t,\qquad
&\left( \Gamma_{2,2}^t \right)_{ij} 		= \Gamma_{i(j+d)}^t,\\
\fl\left( \Gamma_{3}^t \right)_{ij} 	= \Gamma_{(i+d)(j+d)}^t,\qquad
&\left( \Xi_{1}^t \right)_{i} 					= \Xi_{i}^t,\qquad
&\left( \Xi_{2}^t \right)_{i} 					= \Xi_{i+d}^t,
\end{eqnarray*}
for $i,j,k=1,\ldots,d$. These terms are recollected according to 
\begin{equation*}
\overrightarrow{\Lambda}_{1}^t= \vec(\Lambda_{1}^t),\quad
\overrightarrow{\Lambda}_{4}^t= \vec(\Lambda_{4}^t),\quad
\overrightarrow{\Gamma}_{1}^t= \vec(\Gamma_{1}^t),\quad
\overrightarrow{\Gamma}_{3}^t= \vec(\Gamma_{3}^t),
\end{equation*}	
and
\begin{equation*}
		 \overrightarrow{\Lambda}_{2}^t= \left(\begin{array}{c} \vec(\Lambda_{2,1}^t) \\ \vec(\Lambda_{2,2}^t) \\ \vec(\Lambda_{2,3}^t) \end{array}\right),\quad
		 \overrightarrow{\Lambda}_{3}^t= \left(\begin{array}{c} \vec(\Lambda_{3,1}^t) \\ \vec(\Lambda_{3,2}^t) \\ \vec(\Lambda_{3,3}^t) \end{array}\right),\quad
		 \overrightarrow{\Gamma}_{2}^t= \left(\begin{array}{c} \vec(\Gamma_{2,1}^t) \\ \vec(\Gamma_{2,2}^t) \end{array}\right).
\end{equation*}
This reordering enhances the zero pattern of the right hand side of the ordinary differential equation.
		
\begin{theorem}[Schr\"odinger Hamiltonians]\label{theo:ham_schro}
Let $V:\R^d\to\R$ be a smooth function of subquadratic growth and $h:\R^{2d}\to\R$, $h(q,p)=\frac12|p|^2+V(q)$. Let $(q_0,p_0)\in\R^{2d}$ and $(\Phi^t,\vec(\Lambda^t),\vec(\Gamma^t),\Xi^t)$ be the solution of the ordinary differential equation \eref{vectorisation_odes_1} given in Corollary~\ref{thm_vec_odes} with initial values
\begin{equation*}
(\Phi^0,\vec(\Lambda^0),\vec(\Gamma^0),\Xi^0) = ((q_0,p_0),0,0,0)\in\R^{2d+8d^3+4d^2+2d}.
\end{equation*}
We set
\begin{equation*}
\fl \Psi^t_1 = \Phi^t_q\in\R^d,\qquad
\Psi^t_2 = 
\left(\begin{array}{c} \Phi^t_p \\ \overrightarrow{\Lambda}_{2}^t \\ \overrightarrow{\Lambda}_{4}^t \\ 
\overrightarrow{\Gamma}_{2}^t \\ \Xi^t_2  	\end{array}\right)\in\mathbb{R}^{4d^3+2d^2+2d},\qquad
\Psi^t_3 = \left(\begin{array}{c}  \overrightarrow{\Lambda}_{1}^t \\ \overrightarrow{\Lambda}_{3}^t \\ \overrightarrow{\Gamma}_{1}^t \\ \overrightarrow{\Gamma}_{3}^t \\ \Xi^t_1  	\end{array}\right)\in\mathbb{R}^{4d^3+2d^2+d}. 
\end{equation*}
Then, 
	\begin{equation}\label{ode:hamschro}
		\fl\qquad\frac{\rmd}{\rmd t} \left(\begin{array}{c} \Psi^t_1 \\ \Psi^t_2 \\ \Psi^t_3 \end{array}\right) =
			\left(\begin{array}{ccc} 
				0 & A_1 & 0 \\
				0 & 0 & A_2\circ\Psi_1^t \\
				0 & A_3\circ\Psi_1^t & 0
			\end{array} \right)
			\cdot \left(\begin{array}{c} \Psi^t_1 \\ \Psi^t_2 \\ \Psi^t_3 \end{array}\right)
			+ \left(\begin{array}{c} 0 \\ b_2\circ\Psi_1^t \\ 0 \end{array}\right)
	\end{equation}
with $\Psi^0_1=q_0$, $\Psi^0_2 = (p_0,0)$ and $\Psi^0_3= 0$, 
where $A_1 = \left(\begin{array}{ccccc} \Id_d & 0 & 0 & 0 & 0 \end{array}\right)$ and
\begin{equation*}
A_2 = \left(\begin{array}{ccccc} 
				0 & 0 & 0 & 0 & 0 \\
				K_2 & K_3 & 0 & 0 & 0 \\
				0 & K_6 & 0 & 0 & 0 \\
				K_8 & 0 & K_9 & K_{15} & 0 \\
				K_{12} & 0 & K_{13} & 0 &  K_{14}				
		\end{array}\right), \qquad
b_2 = \left(\begin{array}{c}		-DV \\ 0 \\ -\vec(\widetilde{D^3V}) \\ 0 \\ 0 \end{array}\right), 	 
\end{equation*}
and 
\begin{equation*}
A_3 = \left(\begin{array}{ccccc} 
				0 & K_1 & 0 & 0 & 0 \\
				0 & K_4 & K_5 & 0 & 0 \\
				0 & 0 & 0 & K_7 & 0 \\
				0 & K_{10} & 0 & K_{11} & 0 \\
				0 & 0 & 0 & 0 & \Id_d				
		\end{array}\right). 
\end{equation*}
Denoting
\begin{eqnarray*}
		(D^3V)_{\rm m} = 
				\left(\begin{array}{cccccc} 
					 D^3V_{(1,1,1)} & \cdots & D^3V_{(1,1,d)} & D^3V_{(1,2,1)} & \cdots & D^3V_{(1,d,d)} \\
					 \vdots & \vdots & \vdots & \vdots & \vdots & \vdots 	\\					 
					 D^3V_{(d,1,1)} & \cdots & D^3V_{(d,1,d)} & D^3V_{(d,2,1)} & \cdots & D^3V_{(d,d,d)}
				\end{array}\right) 
\end{eqnarray*}
and 
\begin{eqnarray*}
\fl (D^4V)_{\rm m} =\\\vspace*{4em}
\fl\left(\begin{array}{ccccccccc} 
					 D^4V_{(1,1,1,1)} & \cdots & D^4V_{(1,1,1,d)} & D^4V_{(1,1,2,1)} & \cdots & D^4V_{(1,1,d,d)} & D^4V_{(1,2,1,1)} & \cdots & D^4V_{(1,d,d,d)}\\
					 \vdots & \vdots & \vdots & \vdots & \vdots & \vdots & \vdots & \vdots & \vdots 	\\					 
					 D^4V_{(d,1,1,1)} & \cdots & D^4V_{(d,1,1,d)} & D^4V_{(d,1,2,1)} & \cdots & D^4V_{(d,1,d,d)} & D^4V_{(d,2,1,1)} & \cdots & D^4V_{(d,d,d,d)}
				\end{array}\right), 
	\end{eqnarray*}

\bigskip\noindent	
the matrices $K_1,\ldots,K_{15}$ are given by 
	\begin{eqnarray*}		
	\fl	 K_1 = \left(\begin{array}{ccc} \Id_{d^3} & \Id_{d^3} & \Id_{d^3} \end{array}\right),
		 K_2 = \left(\begin{array}{c} -D^2V \otimes \Id_d \otimes \Id_d \\ \Id_d \otimes -D^2V \otimes \Id_d 
								\\ \Id_d \otimes\ \Id_d \otimes -D^2V 
							\end{array}\right),\\
							\fl K_3 = \left(\begin{array}{ccc} 0 & \Id_{d^3} & \Id_{d^3} \\ \Id_{d^3} & 0 & \Id_{d^3} \\ \Id_{d^3} & \Id_{d^3} & 0 \end{array}\right) ,\\
\fl							K_4 = \left(\begin{array}{ccc} 0 & \Id_d \otimes \Id_d \otimes -D^2V & \Id_d \otimes -D^2V \otimes \Id_d \\
													\Id_d \otimes \Id_d \otimes -D^2V & 0 & -D^2V \otimes \Id_d \otimes \Id_d \\
													\Id_d \otimes -D^2V \otimes \Id_d  & -D^2V \otimes \Id_d \otimes \Id_d & 0
							\end{array}\right),\\
\fl							K_5 = \left(\begin{array}{c} \Id_{d^3} \\ \Id_{d^3} \\ \Id_{d^3} \end{array}\right),
		\\ \fl K_6 = \left(\begin{array}{ccc} -D^2V \otimes \Id_d \otimes \Id_d & \Id_d \otimes -D^2V \otimes \Id_d 
								& \Id_d \otimes \Id_d \otimes -D^2V 
							\end{array}\right),	
		\\ \fl K_7 = \left(\begin{array}{cc} \Id_{d^2} & \Id_{d^2} \end{array}\right),
		 K_8 = \left(\begin{array}{c} 0 \\ -(D^3V)_{\rm m} \otimes \Id_d \end{array}\right),
		 K_9 = \left(\begin{array}{c} \Id_d \otimes -D^2V \\ -D^2V \otimes \Id_d \end{array}\right),\\
		\fl  K_{10} = \left(\begin{array}{ccc} 0 & 0 & -(D^3V)_{\rm m} \otimes \Id_d \end{array}\right),
		K_{11} = \left(\begin{array}{cc} -D^2V \otimes \Id_d &  \Id_d \otimes -D^2V\end{array}\right)	,\\	
		 \fl K_{12} = -(D^4V)_{\rm m},
		 K_{13} = -3(D^3V)_{\rm m},
		 K_{14} = -(D^3V)_{\rm m},
		 K_{15} = \left(\begin{array}{c} \Id_{d^2}  \\ \Id_{d^2}\end{array}\right).
	\end{eqnarray*}
\end{theorem}

\begin{proof}
In the following, we denote 
\begin{equation*}
A_{i_1,\ldots,s,\ldots,i_k}B_{j_1,\ldots,s,\ldots,j_l} = \sum_{s=1}^d A_{i_1,\ldots,s,\ldots,i_k}B_{j_1,\ldots,s,\ldots,j_l}
\end{equation*}
for tensors $A\in\R^{d\times \cdots \times d}$ and $B\in\R^{2d\times \cdots\times 2d}$. By the specific form of the matrix $M^t$ given in \eref{eq:MC}, we obtain 
	\begin{eqnarray*}
		M^t_{i l_1} \Lambda^t_{l_1 l_2 l_3} = 
			\cases{ 
				\Lambda^t_{(i+d)l_2l_3}, & for $i \leq d$, \\
				(-D^2V\circ \Phi^t_q )_{(i-d)s} \cdot \Lambda^t_{s l_2 l_3}, & else,}\\
		M^t_{i l_2} \Lambda^t_{l_1 l_2 l_3} = 
			\cases{ 
				\Lambda^t_{l_1(i+d)l_3}, & for $i \leq d$, \\
				(-D^2V\circ \Phi^t_q )_{(i-d)s} \cdot \Lambda^t_{l_1 s l_3}, & else,}\\
		M^t_{i l_3} \Lambda^t_{l_1 l_2 l_3} = 
			\cases{ 
				\Lambda^t_{l_1l_2(i+d)}, & for $i \leq d$, \\
				(-D^2V\circ \Phi^t_q )_{(i-d)s} \cdot \Lambda^t_{l_1 l_2 s}, & else,}		
	\end{eqnarray*}
Therefore, 	
\begin{eqnarray*}
		\frac{\rmd}{\rmd t}\Lambda^t_{ijk} 
			&=& 1_{\{i \leq d\}} \Lambda^t_{(i+d) j k} 
			- 1_{\{i > d\}} (D^2V\circ \Phi^t_q)_{(i-d)s} \cdot \Lambda^t_{s j k}\\
			&&+ 1_{\{j \leq d\}} \Lambda^t_{i (j+d) k} 
			- 1_{\{j > d\}} (D^2V\circ \Phi^t_q)_{(j-d)s} \cdot \Lambda^t_{isk}\\
			&&+ 1_{\{k \leq d\}} \Lambda^t_{i j (k+d)} 	
			- 1_{\{k > d\}} (D^2V\circ \Phi^t_q)_{(k-d)s} \cdot \Lambda^t_{i j s}\\
			&&- 1_{\{i,j,k > d\}} \widetilde{D^3V}_{(i-d)(j-d)(k-d)} \circ \Phi^t_q,
\end{eqnarray*}	
since 		
\begin{eqnarray*}
C^1_{ijk}(t)  = \cases{ -\widetilde{D^3V}_{(i-d)(j-d)(k-d)} \circ \Phi^t_q, & $i,j,k>d$ \\ 0, & else }.
\end{eqnarray*}
In the same way, we obtain
\begin{eqnarray*}			
	\frac{\rmd}{\rmd t}\Gamma^t_{ij}
			&=& 1_{\{i \leq d\}} \Gamma^t_{(i+d) j} - 1_{\{i > d\}} (D^2V\circ \Phi^t_q)_{(i-d)s} \cdot \Gamma^t_{sj}\\
					&&+ 1_{\{j \leq d\}} \Gamma^t_{i (j+d)} - 1_{\{j > d\}} (D^2V\circ \Phi^t_q)_{(j-d)s} \cdot \Gamma^t_{is}\\
				  &&- 1_{\{i \leq d\}} (D^3V\circ \Phi^t_q)_{(i-d)s s'} \Lambda_{s' s j},
\end{eqnarray*}
since
\begin{equation*}
\left( C^2_i(t) \right)_{jk} = \cases{ -(D^3V)_{(i-d)jk} \circ \Phi^t_q, & $i>d, j,k \leq d$ \\  0, & else }.
\end{equation*}
Finally, we compute				
\begin{eqnarray*}
\fl\frac{\rmd}{\rmd t}\Xi^t_{i}
			= 1_{\{i \leq d\}} \Xi^t_{(i+d)} - 1_{\{i > d\}} (D^2V\circ \Phi^t_q)_{(i-d)s} \cdot \Xi^t_{s}
				- 1_{\{i > d\}} (D^4V\circ \Phi^t_q)_{i s s' s''} \Lambda^t_{s'' s' s}
				\\- 3\cdot 1_{\{i > d\}} (D^3V\circ \Phi^t_q)_{i s s'} \Gamma^t_{s' s},
	\end{eqnarray*}
using that
\begin{eqnarray*}
		\left( C^3_i(t) \right)_{jkl} = \cases{ -(D^4V)_{(i-d)jkl} \circ \Phi^t_q, & $i>d, j,k,l \leq d$ \\  0, & else } .  
\end{eqnarray*}
Reordering the above differential equations, we obtain the claimed result.
\end{proof}

\section{Discretization}\label{discretization}

We adopt the following general scheme which has been previously developed for discretizing Egorov's theorem~\cite{LR10}: The classical Hamiltonian flow $\Phi^t$ is discretized by a symplectic order~$p$ method $z_{n+1}=\Psi^\tau(z_n)$ with sufficiently small step size $\tau>0$. The initial Wigner function is split into its positive and negative part
$W_{\psi_0} = W_{\psi_0}^+ - W_{\psi_0}^-$  
and is sampled by sufficiently many phase space points $z_1^\pm,\ldots,z_N^\pm\in\R^{2d}$.Then, expectation values for various observables $a:\R^{2d}\to\R$ with respect to the solution $\psi_t$ of the Schr\"odinger equation
\begin{equation*}
\rmi\eps\partial\psi_t = H\psi_t
\end{equation*}
are approximated according to 
\begin{equation}\label{algorithm}
\eqalign{
\langle\psi_t,\op(a)\psi_t\rangle = \langle\psi_0,\rme^{\rmi Ht/\eps}\op(a)\rme^{-\rmi Ht/\eps}\psi_0\rangle\\
= \langle\psi_0,\op(a\circ\Phi^t)\psi_0\rangle + O(\eps^2)
= \int_{\R^{2d}} a(\Phi^t(z)) W_{\psi_0} \rmd z + O(\eps^2)\\
\approx 
\frac1N \sum_{j=1}^N a((\Psi^\tau\circ\cdots\circ\Psi^\tau)(z_j^+)) - 
 \frac1N \sum_{j=1}^N a((\Psi^\tau\circ\cdots\circ\Psi^\tau)(z_j^-)) \\
=: I^N(a\circ\Psi^\tau\circ\cdots\circ\Psi^\tau).}
\end{equation}
The computational work of this algorithm lies in the sampling of the initial Wigner function and the classical evolution of the sample points. Then, expectation values are computed by a final phase space summation.

\begin{remark}
If the initial wave function $\psi_0$ is a Gaussian wave packet, then the Wigner function $W_{\psi_0}$ is positive, see 
\cite[Theorem~1.102]{Fo89}, and the approximation \eref{algorithm} reads as
$$
\langle\psi_t,\op(a)\psi_t\rangle \approx \frac1N \sum_{j=1}^N a((\Psi^\tau\circ\cdots\circ\Psi^\tau)(z_j))
$$
with $z_1,\ldots,z_N\in\R^{2d}$ sampled according to $W_{\psi_0}$. If $\psi_0$ is a superposition of Gaussian wave packets, then the Wigner function $W_{\psi_0}$ is a sum of phase space Gaussians plus oscillatory cross terms, such that stratified sampling can be applied, see \cite[\S3]{LR10}.
\end{remark}

\subsection{Splitting the two integrals}
Here, we add the second order correction $\eps^2 a_2(t)$ to $a_0(t)=a\circ\Phi^t$ and discretize 
\begin{equation*}
\fl\left\langle \op(a_0(t)+\eps^2 a_2(t))\psi_0, \psi_0 \right\rangle =
\int_{\mathbb{R}^{2d}} a_0(t,z) W_{\psi_0}(z)\rmd z +\eps^2 \int_{\mathbb{R}^{2d}} a_2(t,z) W_{\psi_0}(z)\rmd z,
\end{equation*}
where we split the phase space integral into two parts. This splitting has an impact on the computing time, since the prefactor $\eps^2$ allows to discretize the second summand rather coarsely without diminishing the overall accuracy of the approximation. We approximate the two integrals via
\begin{equation*}
 		\int_{\mathbb{R}^{2d}} a_j(t,z) W_{\psi_0}(z)\rmd z \;\approx\; 
 			I^{N_j} \left(a_j(t) \right),\qquad j=0,2,
\end{equation*}
such that
$\left\langle \op(a_0(t)+\eps^2 a_2(t))\psi_0,\psi_0  \right\rangle 
			\approx I^{N_0} \left(a_0(t) \right) + \eps^2 I^{N_2}\left(a_2(t) \right)$.

\subsection{Computing the integrals}\label{sec_discr_phase_sampl}
The two integrals depend on the Wigner function of the initial wave function. Here, we consider a Gaussian wave packet centered at $(q_0,p_0)\in \mathbb{R}^{2d}$, 
\begin{equation}\label{intro_wave_pack}
	\psi_0(q)= (\pi\eps)^{-d/4} \exp\!\left(-(2\eps)^{-1}|q-q_0|^2+ \frac{\rmi}{\eps}\,p_0 \cdot (q-q_0)\right).
\end{equation}
In this case the Wigner function can be calculated analytically as a phase space Gaussian with mean $(q_0,p_0)$ and covariance matrix $\Id_{2d}$ (see, for example \cite[\S3]{LR10}), namely 
\begin{eqnarray*}
 W_{\psi_0}(z)= (\pi \epsilon)^{-d} \exp\!\left(-\frac{1}{\eps} |z-(q_0,p_0)|^2\right).
\end{eqnarray*}
Let $f:\mathbb{R}\times \mathbb{R}^{2d}\to\R$ be the function to be integrated. Due to the high dimensionality of the problem, we use quasi-Monte Carlo quadrature. That is, 
\begin{eqnarray*}
	\int_{\mathbb{R}^{2d}} f(t,z) W_{\psi_0}(z)\rmd z
	\approx
		I^{N}(f(t)) = \frac{1}{N}\sum_{j=1}^{N}f(t,z_j),
\end{eqnarray*}
with quadrature nodes $\{ z_j \}^N_{j=1}\subset \R^{2d}$ of low star discrepancy with respect to the multivariate normal distribution. 
Then, the Koksma-Hlawka inequality yields a constant $\gamma=\gamma(f(t))>0$ such that
\begin{eqnarray} \label{approx_error_phase_sampl}
	\left| \int_{\mathbb{R}^{2d}} f(t,z) W_{\psi_0}(z)\rmd z	- I^{N}(f(t)) \right| \leq \gamma\, (\log N)^{c_d} N^{-1},
\end{eqnarray} 
where $c_d \geq 2d$,see, for example, \cite[\S3.2]{LR10}.

\subsection{Splitting the ordinary differential equations}\label{sec_discrete_a2}
To compute $a_0(t)=a\circ\Phi^t$ we have to discretize the Hamiltonian equation
\begin{equation*}
\frac{\rmd}{\rmd t}\left( \begin{array}{c} \Phi^t_q \\ \Phi^t_p \end{array} \right) = 
			\left( \begin{array}{cc} 0 & \Id_d \\ 0 & 0 \end{array} \right) \left( \begin{array}{c} \Phi^t_q \\ \Phi^t_p \end{array} \right) + \left( \begin{array}{c} 0 \\ DV \circ \Phi^t_q \end{array} \right). 
\end{equation*}
Now let $\phi_1^t$ and $\phi_2^t$ be the flows of the following differential equations
\begin{eqnarray*} 
	\frac{\rmd}{\rmd t} y(t) = \left( \begin{array}{cc} 0 & \Id_d \\ 0 & 0 \end{array} \right) \cdot y(t), \qquad 
	 \frac{\rmd}{\rmd t} y(t) = \left( \begin{array}{c} 0 \\ DV(y_1(t)) \end{array} \right).		
\end{eqnarray*}
These flows can be computed exactly by
\begin{eqnarray*}
	 \phi_1^t(y) = y + t\left( \begin{array}{c} y_2 \\ 0\end{array} \right) ,
	 \qquad \phi_2^t = y + t\left( \begin{array}{c} 0\\ DV(y_1) \end{array} \right),
\end{eqnarray*}
and
\begin{equation*}
\phi^\tau= \phi^{\tau/2}_1 \circ \phi^{\tau}_2 \circ \phi^{\tau/2}_1
\end{equation*}
defines a symplectic second order splitting scheme for $\tau>0$, the so-called Strang splitting, 
see \cite[\S4.3]{HWL06}. 
By suitable compositions of this scheme, one can construct symplectic splitting schemes of arbitrary order, see e.g.  \cite[\S4]{Yo90}).

For computing the correction term $a_2(t)$, we write the ordinary differential equation~\eref{ode:hamschro} of Theorem~\ref{theo:ham_schro} as
\begin{equation}\label{ode:split}
\frac{\rmd}{\rmd t}\Psi^t = N_1\Psi^t + \left((N_2\circ\Psi_1^t) \Psi^t + (B_2\circ\Psi_1^t)\right) + (N_3\circ\Psi_1^t)\Psi^t
\end{equation}
with
\begin{eqnarray*}
N_1 = \left(\begin{array}{ccc} 
				0 & A_1 & 0 \\
				0 & 0 & 0 \\
				0 & 0 & 0
			\end{array}\right),\qquad
N_2	=	\left(\begin{array}{ccc} 
				0 & 0 & 0 \\
				0 & 0 & A_2 \\
				0 & 0 & 0
			\end{array}\right),\qquad
B_2 = \left(\begin{array}{c} 0 \\ b_2 \\ 0 \end{array}\right),
\\
N_3 = \left(\begin{array}{ccc} 
				0 & 0 & 0 \\
				0 & 0 & 0 \\
				0 & A_3 & 0
			\end{array}\right).
\end{eqnarray*}
The zero pattern of the matrices $N_1,N_2,N_3$ allows to compute explicit flow maps. Indeed, 
let $\psi^t_1$, $\psi^t_2$, $\psi^t_3$ be the flows of the differential equations 
\begin{eqnarray*}
	\frac{\rmd}{\rmd t} y(t) = N_1 \cdot y(t) , \qquad
	\frac{\rmd}{\rmd t} y(t) = N_2(y(t)) \cdot y(t) + B_2(y(t)), \\
	\frac{\rmd}{\rmd t} y(t) = N_3(y(t)) \cdot y(t).
\end{eqnarray*}
Then, 
\begin{eqnarray*}
	\fl \psi^t_1(y) = y +  
		t\left(\begin{array}{c}
		 A_1 \cdot y_2 \\0  \\ 0 
		\end{array}\right) , \qquad
		\psi^t_2(y) = y +  
		t\left(\begin{array}{c}
			0 \\ A_2(y_1) \cdot y_3 + b_2(y_1)\\ 0
		\end{array}\right) ,
	\\\psi^t_3(y) = y + 
		t\left(\begin{array}{c}
			0  \\ 0 \\ A_3(y_1) \cdot y_2
		\end{array}\right)  .
\end{eqnarray*}    
For $\tau>0$, we obtain a first order splitting scheme for equation~\eref{ode:split} by 
\begin{eqnarray*}
	\psi^\tau= \psi^{\tau}_2 \circ \psi^{\tau}_1 \circ \psi^{\tau}_3,
\end{eqnarray*}
see, \cite[\S2.5]{HWL06}. Its adjoint is $\psi^{\tau*}= \psi^{\tau}_3 \circ \psi^{\tau}_1 \circ \psi^{\tau}_2$, 
and we can create a second order symmetric splitting method by
\begin{eqnarray*}
	F_2^\tau = \psi^{\tau/2} \circ \psi^{\tau/2*} =  \psi^{\tau/2}_2 \circ \psi^{\tau/2}_1 \circ \psi^{\tau}_3 \circ \psi^{\tau/2}_1 \circ \psi^{\tau/2}_2.
\end{eqnarray*}
A corresponding fourth order splitting is obtained by 
\begin{equation*}
F_4^\tau = F_2^{\frac{\tau}{2-2^{1/3}}} \circ F_2^{\frac{-2^{1/3} \tau}{2-2^{1/3}}} \circ F_2^{\frac{\tau}{2-2^{1/3}}}, 
\end{equation*}
see \cite[\S4]{Yo90}, and we obtain $\Psi^\tau-F_4^\tau = O(\tau^5)$ as $\tau\to0$. By construction, the first $2d$ components of $\psi^\tau, F_2^\tau$, and $F_4^\tau$, respectively,  define symplectic maps on phase space. 
\subsection{The approximation scheme }\label{sec_discrete_param}
At this point we have built up all the ingredients for computing an approximation to expectation values, 
which is fourth order accurate with respect to the semiclassical parameter $\eps$. Let $a:\R^{2d}\to\R$ be a Schwartz function. According to the discussion in Section~\ref{sec_discrete_a2}, we use time splitting schemes of order eight and order four, for the approximation of $a_0(t)$ and $a_2(t)$, respectively. That is, 
\begin{equation*}
a_0(t) = \widetilde a_0^\tau(t) + O(\tau^8),\qquad a_2(t) = \widetilde a_2^\tau(t) + O(\tau^4),
\end{equation*}
as $\tau\to0$, with $\widetilde a_0^\tau(t) = a\circ(\phi^\tau\circ\cdots\circ\phi^\tau)$ and 
\begin{eqnarray*}
\widetilde a_2(\tau) = -\case14 && \left( \left(D^3a\circ(\widetilde\Phi^\tau\circ\cdots\circ\widetilde\Phi^\tau)\right)_{ijk}(\widetilde\Lambda^\tau_{kji}\circ\cdots\circ\widetilde\Lambda^\tau_{kji})\right. \\
&& + 3\left(D^2a\circ(\widetilde\Phi^\tau\circ\cdots\circ\widetilde\Phi^\tau)\right)_{ij}
(\widetilde\Gamma^\tau_{ji}\circ\cdots\circ\widetilde\Gamma^\tau_{ji})\\ 
&& + \left(Da\circ(\widetilde\Phi^\tau\circ\cdots\circ\widetilde\Phi^\tau)\right)_{i}
(\widetilde\Xi^\tau_{i}\circ\cdots\circ\widetilde\Xi^\tau_i)\Big),
\end{eqnarray*}
where $\phi^\tau$ is a symplectic eighth order splitting for the Hamiltonian flow $\Phi^\tau$, while $\widetilde\Phi^\tau$ and the tensors $\widetilde\Lambda^\tau, \widetilde\Gamma^\tau, \widetilde\Xi^\tau$ consist of the appropriate components of the fourth order splitting $F^\tau_4$.
Using the quasi-Monte Carlo estimate (\ref{approx_error_phase_sampl}), we obtain
\begin{eqnarray*}
\fl\left\langle \op(a_0(t)+\eps^2 a_2(t))\psi_0,\psi_0 \right\rangle
 = I^{N_0}\left(\widetilde a_0^{\tau_0}(t)\right) + O\!\left((\log N_0)^{c_d}N_0^{-1}\right) + O(\tau_0^8)\\
 \qquad\qquad\qquad +\, \eps^2\, I^{N_2}\left( \widetilde a_2^{\tau_2}(t)\right) 
	  + \eps^2\, O\!\left((\log N_2)^{c_d}N_2^{-1}\right) + \eps^2\, O(\tau_2^4)
\end{eqnarray*}
for $\tau_0,\tau_2\to 0$ and $N_0,N_2\to\infty$. Choosing the time steps and the number of sampling points such that
\begin{equation*}
\max\!\left(\tau_0^8,(\log N_0)^{c_d}N_0^{-1}\right) \le \eps^4,\qquad
\max\!\left(\tau_2^4,(\log N_2)^{c_d}N_2^{-1}\right) \le \eps^2,
\end{equation*}
we finally obtain the desired asymptotic approximation
\begin{equation*}
\left\langle \op(a)\psi_t,\psi_t \right\rangle = 
		I^{N_0}\left(\widetilde a_0^{\tau_0}(t)\right) + \eps^2 I^{N_2}\left( \widetilde a_2^{\tau_2}(t)\right)
		+ O(\eps^4).
\end{equation*}
We note, that the number of sampling points $N_2$ can be chosen much smaller than $N_0$ and that the step size $\tau_2$ can be chosen about the same size as $\tau_0$. Ignoring the logarithmic term in the quasi-Monte Carlo estimate, we deduce as a rule of thumb: 
\begin{equation*}
N_2 \approx \eps^2 N_0,\qquad \tau_2 \approx \tau_0.
\end{equation*}
Hence, although the system of differential equations for the approximation of the correction $a_2(t)$ is more intricate than the one for $a_0(t)$, for moderate dimensions $d$, the computation of the correction is less costly.
\section{Numerical Experiments}\label{numerical}

For our numerical experiments, we consider the time dependent Schr\"odinger equation 
\begin{equation*}
\rm i\eps\partial\psi_t = \left(-\frac{\eps^2}{2}\Delta + 2-\cos(q_1)-\cos(q_2)\right)\psi_t
\end{equation*} 
in two dimensions with torsional potential. The time interval is $[0,15]$, and the initial data $\psi_0$ is chosen as a single Gaussian wave packet (\ref{intro_wave_pack}) with center $(q_0,p_0)=(1,0.5,0,0)$. We validate the approach developed in Section~\ref{discretization} for different observables, namely the position and momentum operators given by
\begin{eqnarray*}
	(q,p)\mapsto q_j, \qquad (q,p)\mapsto p_j, \qquad j=1,2
\end{eqnarray*} 
as well as the potential, kinetic and total energy operators defined by
\begin{eqnarray*}
	(q,p)\mapsto V(q), \qquad (q,p)\mapsto \case{1}{2} |p|^2, \qquad (q,p)\mapsto h(q,p)
\end{eqnarray*}
Since the computation of the leading order Egorov term $a_0(t)=a\circ\Phi^t$ has already been elaborated in \cite{LR10}, we will mainly focus on the correction term $a_2(t)$. 
The main goal of our numerical experiments is to show that our suggested algorithm can reach order four accuracy with respect to $\eps$, does this in an efficient way, and thus is feasible in a moderately high dimensional setting. 

As a reference, we use grid-based solutions to the Schr\"odinger equation computed by a Strang splitting scheme with Fourier collocation, see \cite[Appendix]{LR10}, and derive the needed expectation values. The parameters we use for computing the reference solution are given in Table~\ref{tab_ref_data}. 
\begin{table}[htbp]
	\label{tab_ref_data}
	\caption{The discretization parameters for the grid-based reference solutions computed by a Strang splitting scheme with Fourier collocation.  }
	\begin{indented}
	\item[]\begin{tabular}{ @{} l  l  l  l  l  l  l }
		\br
		$\epsilon$ & \# time steps & domain & space grid \\ \mr
		$0.1$ & $1.2*10^5$  & $[-3,3]\times[-3,3]$ & $1024\times1024$ \\ 
		$0.05$ & $3.6*10^5$ & $[-3,3]\times[-3,3]$ & $1024\times1024$ \\ 
		$0.02$ & $3.6*10^5$ & $[-3,3]\times[-3,3]$ & $1024\times1024$ \\ 
		$0.01$ & $1.2*10^6$ & $[-3,3]\times[-3,3]$ & $1024\times1024$ \\ 
		\br
	\end{tabular}
	\end{indented}
\end{table}

\subsection{Time-evolution of observables}

Figure~\ref{fig_vals_all} presents the expectation values computed by the full approximation
\begin{equation}\label{eq_approx_all}
\langle\op(a)\psi_t,\psi_t\rangle \approx I^{N_0}(\widetilde a_0^{\tau_0}(t)) + 
\eps^2 I^{N_2}(\widetilde a_2^{\tau_2}(t))
\end{equation}
for the case $\eps=10^{-2}$. We note that the total energy is preserved, see also Remark~\ref{rem:inv}.
Figure~\ref{fig_errors_all} compares the expectation values computed from the reference solution with those from the Egorov approximation
\begin{equation}\label{eq_approx_a0}
\langle\op(a)\psi_t,\psi_t\rangle \approx I^{N_0}(\widetilde a_0^{\tau_0}(t))
\end{equation}
as well as the corrected approximation~(\ref{eq_approx_all}). We observe that the error of the corrected approach is of the order $\eps^4$, while the error of the uncorrected approximation (\ref{eq_approx_a0}) is only of the order $\eps^2$.  
\begin{figure}[htbp]
  \centering
  \subfigure[]{
    \includegraphics[width=6cm]{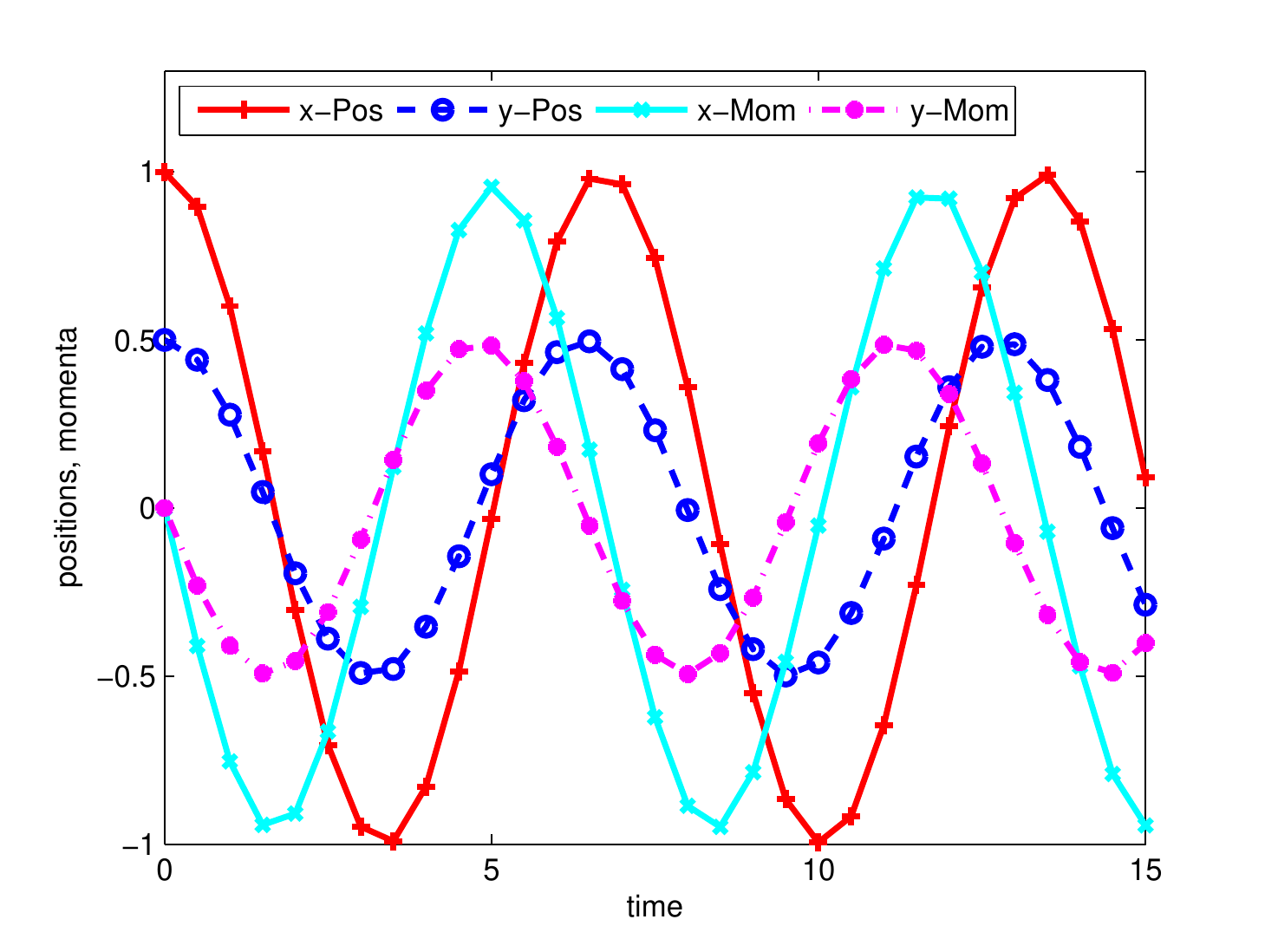} 
  }
  \subfigure[]{
    \includegraphics[width=6cm]{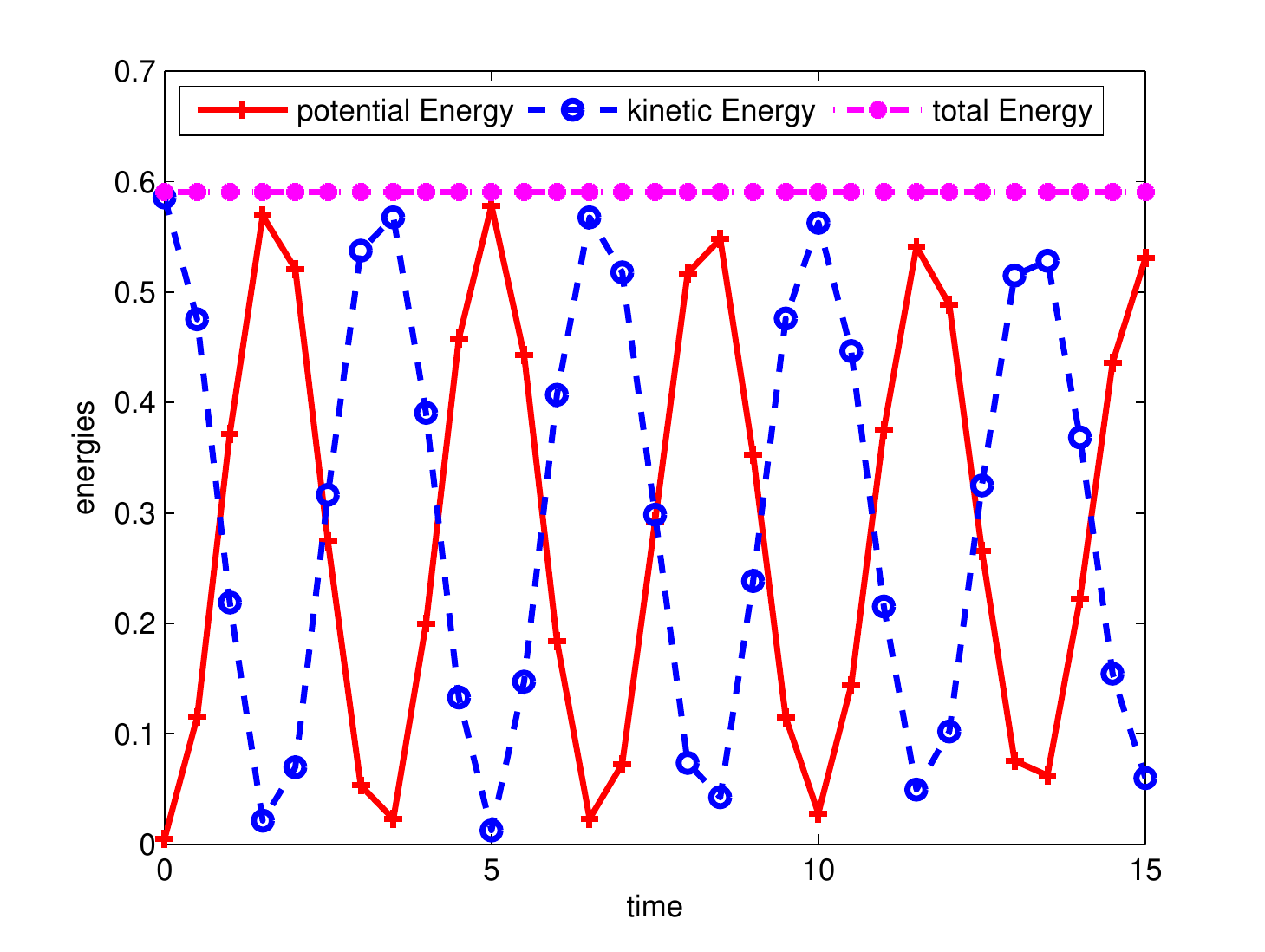}  
  }
  \caption{Figure 1(a) shows the expectation values of positions and momenta and Figure 1(b) the expectation values of kinetic, potential and total energy for the two dimensional torsion potential as a function of time computed by the approximation~(\ref{eq_approx_all}). The semiclassical parameter is chosen as $\eps=10^{-2}$.}
	\label{fig_vals_all}
\end{figure}
\begin{figure}[htbp]
  \centering
  \subfigure[]{
    \includegraphics[width=6cm]{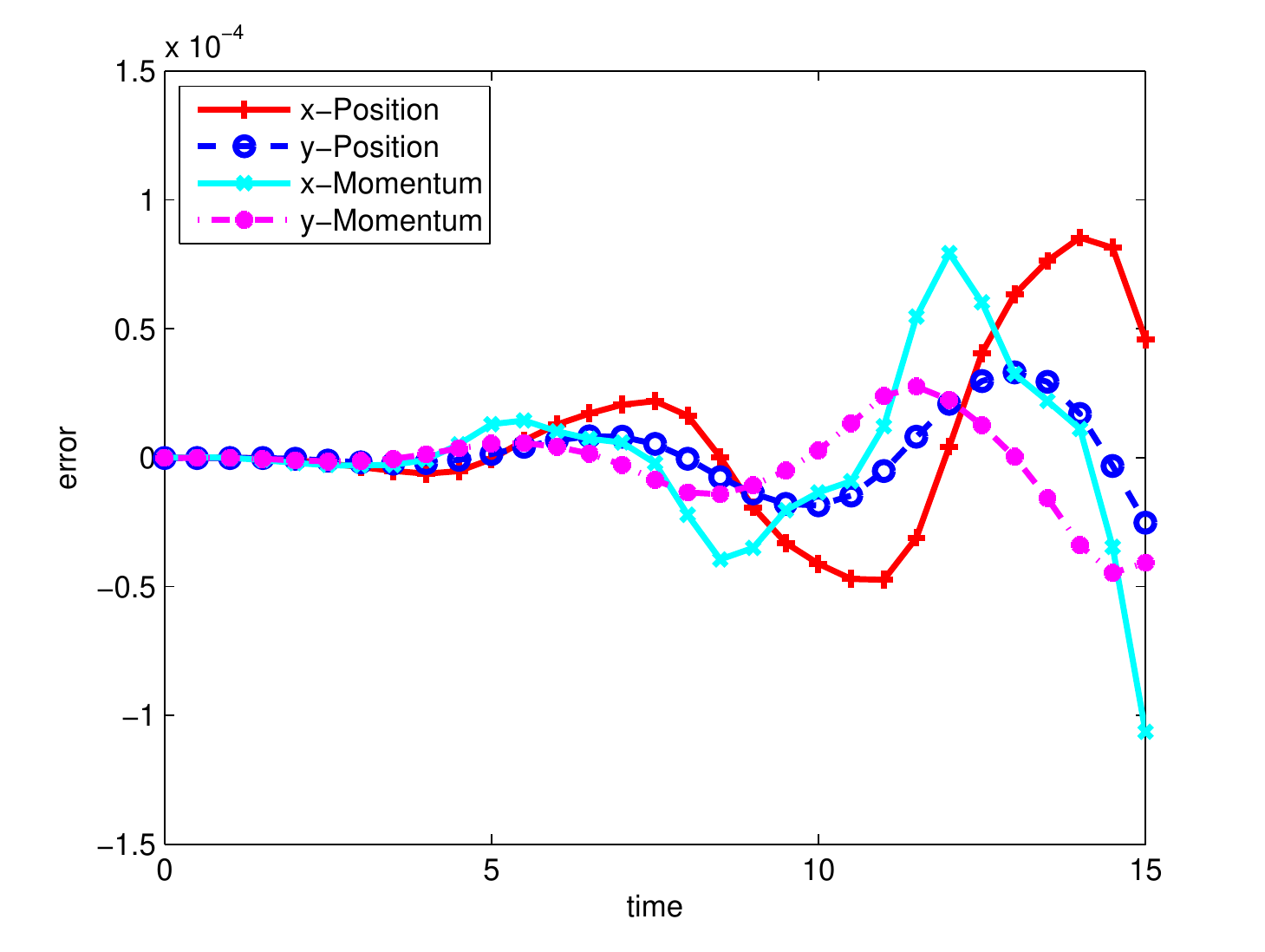} 
  }
  \subfigure[]{
    \includegraphics[width=6cm]{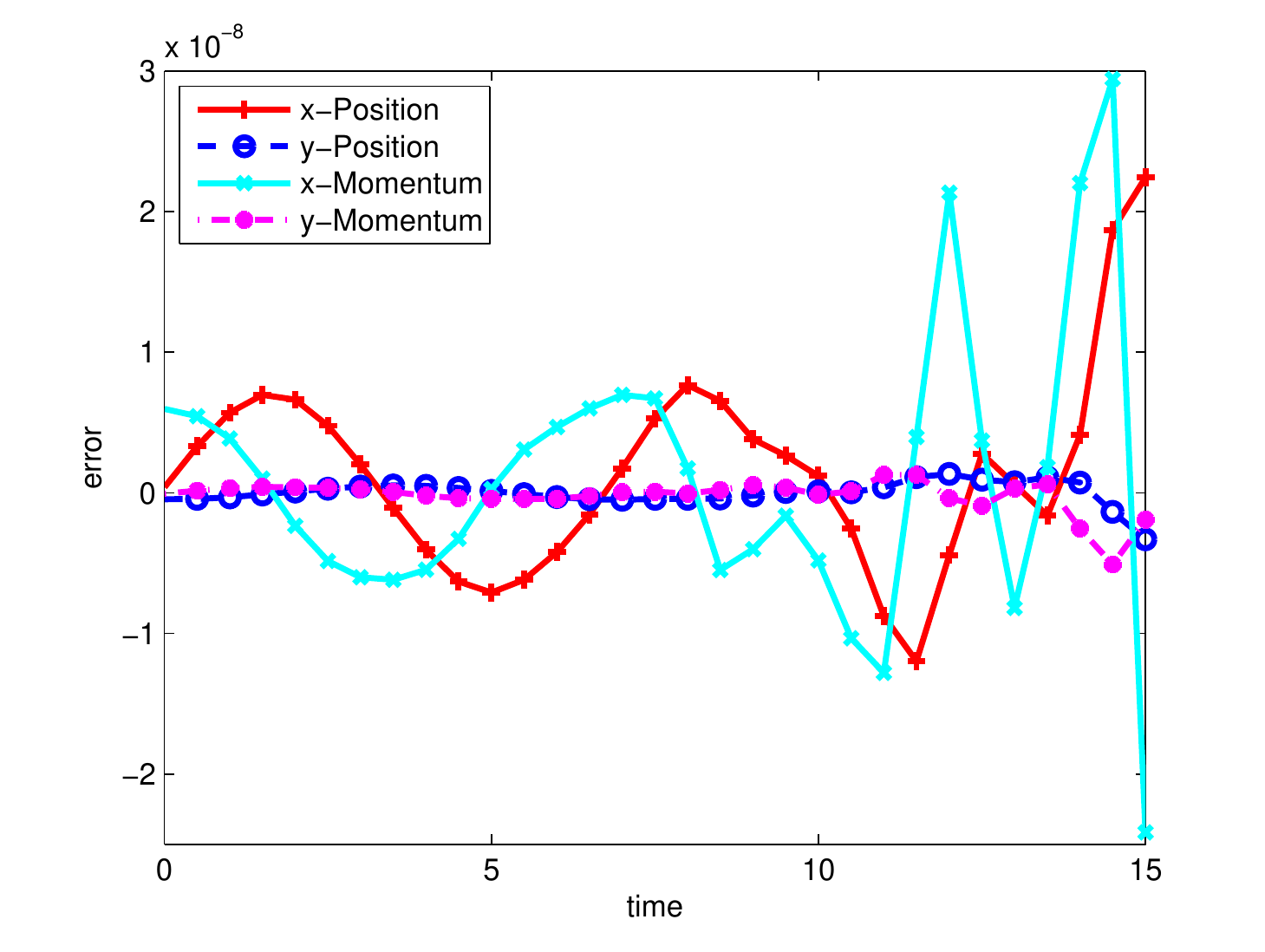}  
  }\\
  \subfigure[]{
    \includegraphics[width=6cm]{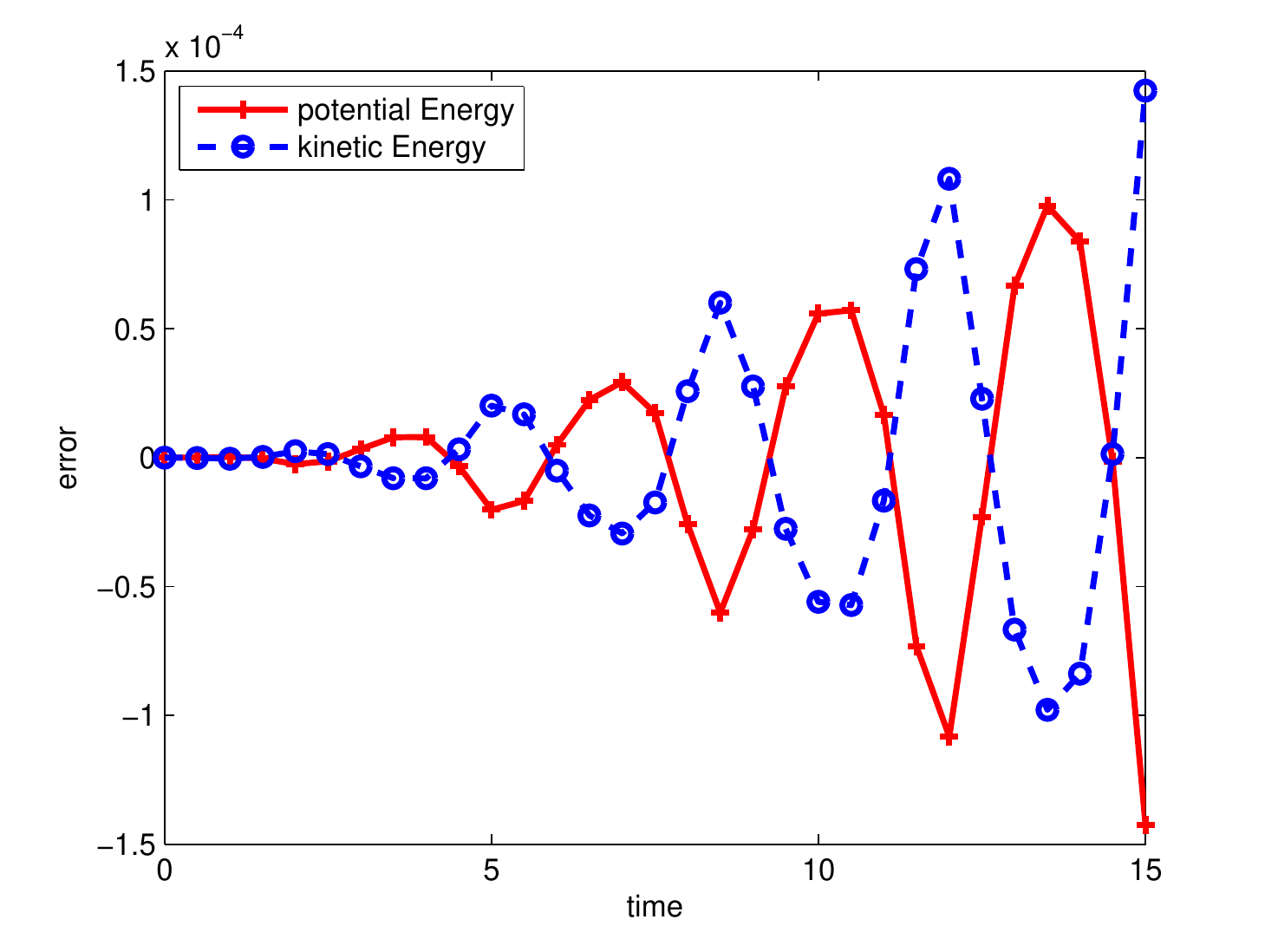} 
  }
  \subfigure[]{
    \includegraphics[width=6cm]{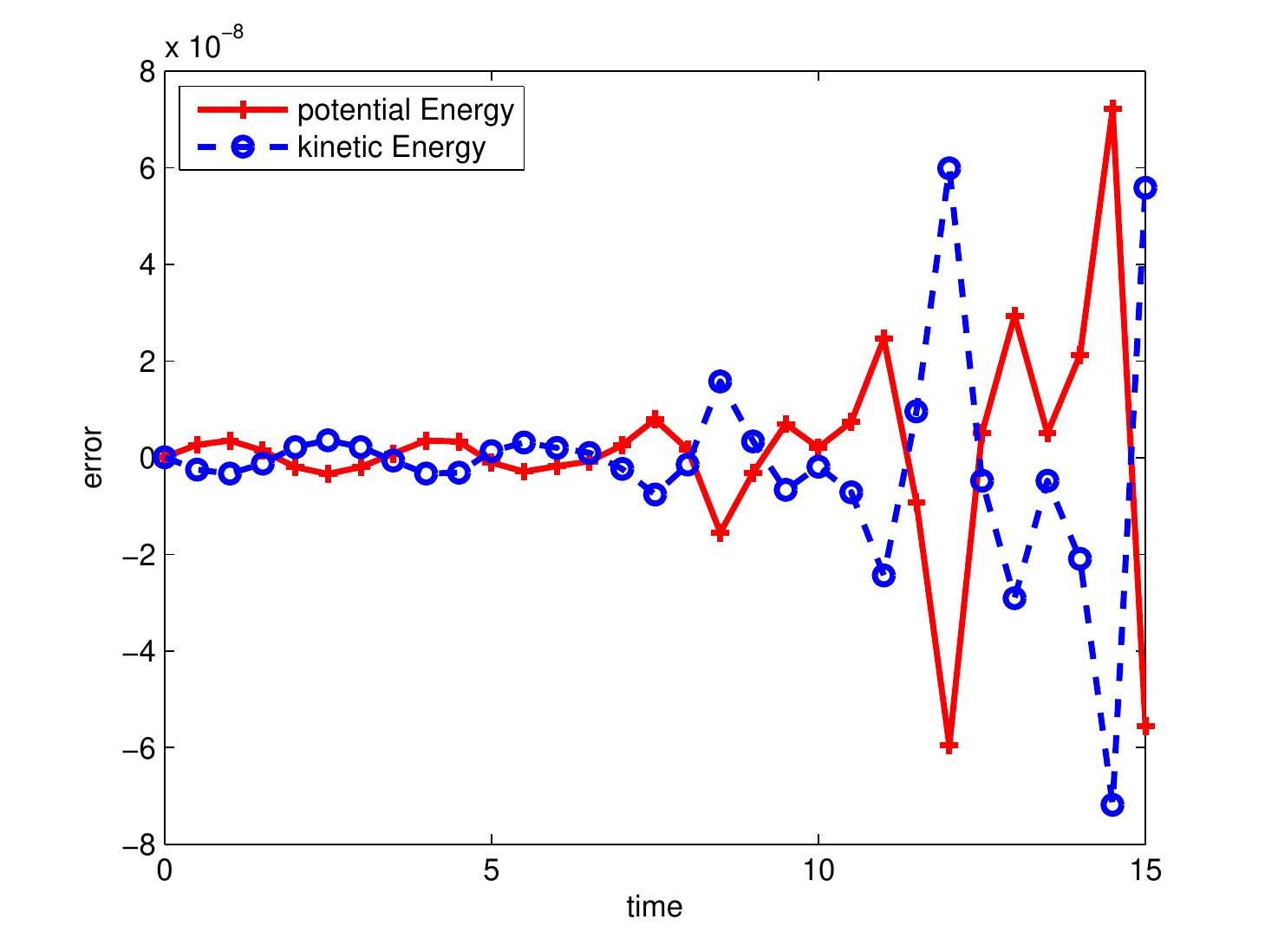}  
  }
  \caption{Figures 2(a) and 2(c) show the absolute error of the expectation values of positions and momenta as well of the kinetic and potential energy, respectively, using the simple Egorov approximation~(\ref{eq_approx_a0}).
	Figures 2(b) and 2(d) show the same for the corrected approximation~(\ref{eq_approx_all}).	
	The semiclassical parameter is chosen as $\eps=10^{-2}$. In the two pictures on the left, the error is of order $\eps^2$, whereas those on the right reach an order of $\eps^4$.  }
  \label{fig_errors_all}
\end{figure}

\subsection{Asymptotic accuracy}
To validate that the corrected algorithm is of the order of four in $\eps$, we compute the maximal and mean deviation from our reference solution for $\eps=0.1,0.05,0.02,0.01$, where the number of Halton points $N_0$ and $N_2$ and the size of the time steps $\tau_0$ and $\tau_2$ are given in Table~\ref{tab_diff_eps}. In Table~\ref{tab_diff_eps} the computing times of the Egorov approximation scheme (\ref{eq_approx_a0}) and the correction term 
\begin{equation}\label{eq_approx_a2} 
\langle \op(a_2(t))\psi_0,\psi_0\rangle \approx I^{N_2}(\widetilde a_2^{\tau_2}(t))
\end{equation}
are presented as well. For a given accuracy we can observe that, especially for small semiclassical parameters, (\ref{eq_approx_a2}) can be computed much faster than (\ref{eq_approx_a0}). In Figure~\ref{fig_diff_eps} the expected fourth order of our approximation can be seen.
\begin{figure}[htbp]
  \centering
  \subfigure[]{
    \includegraphics[width=6cm]{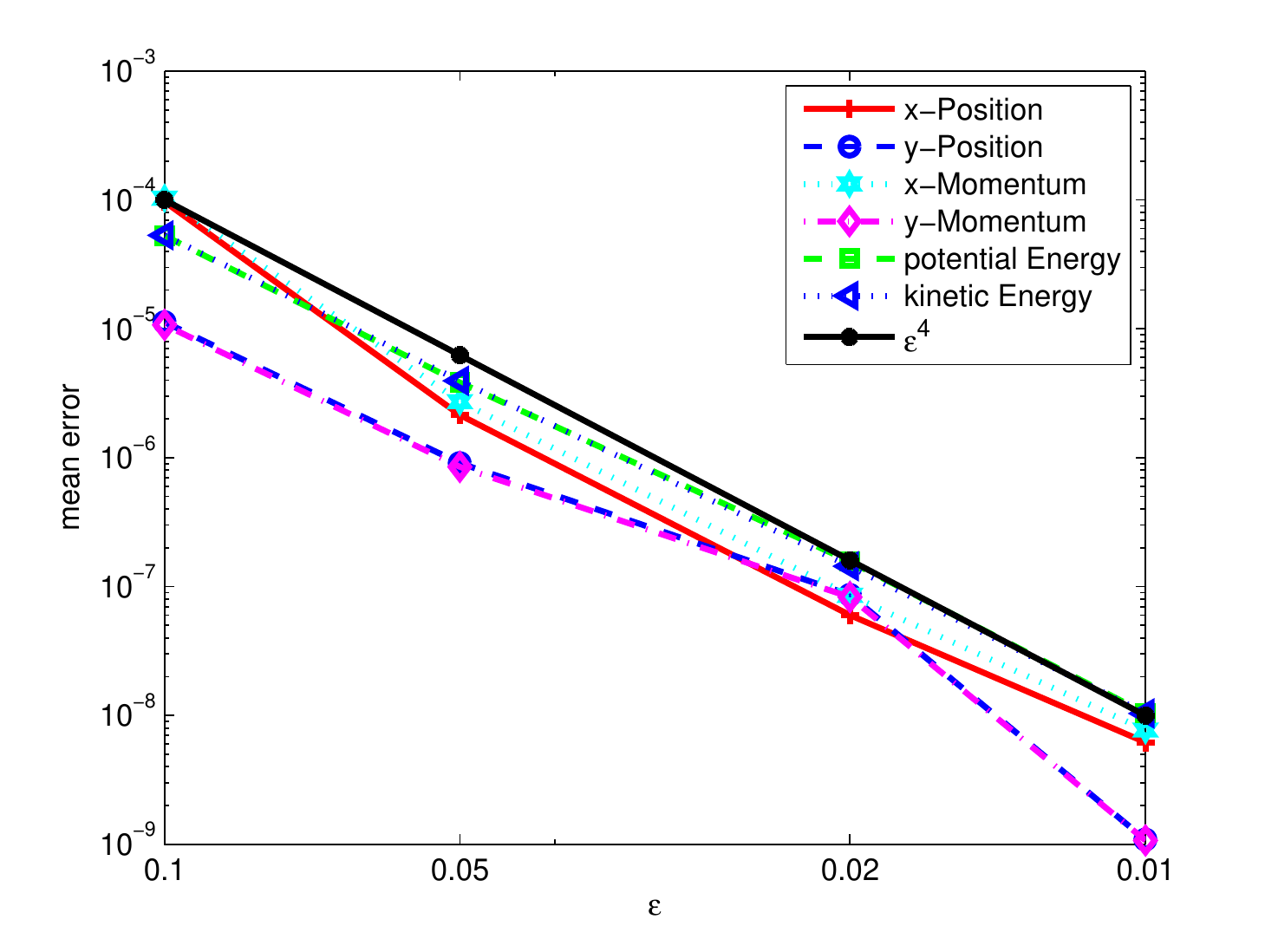} 
  }
  \subfigure[]{
    \includegraphics[width=6cm]{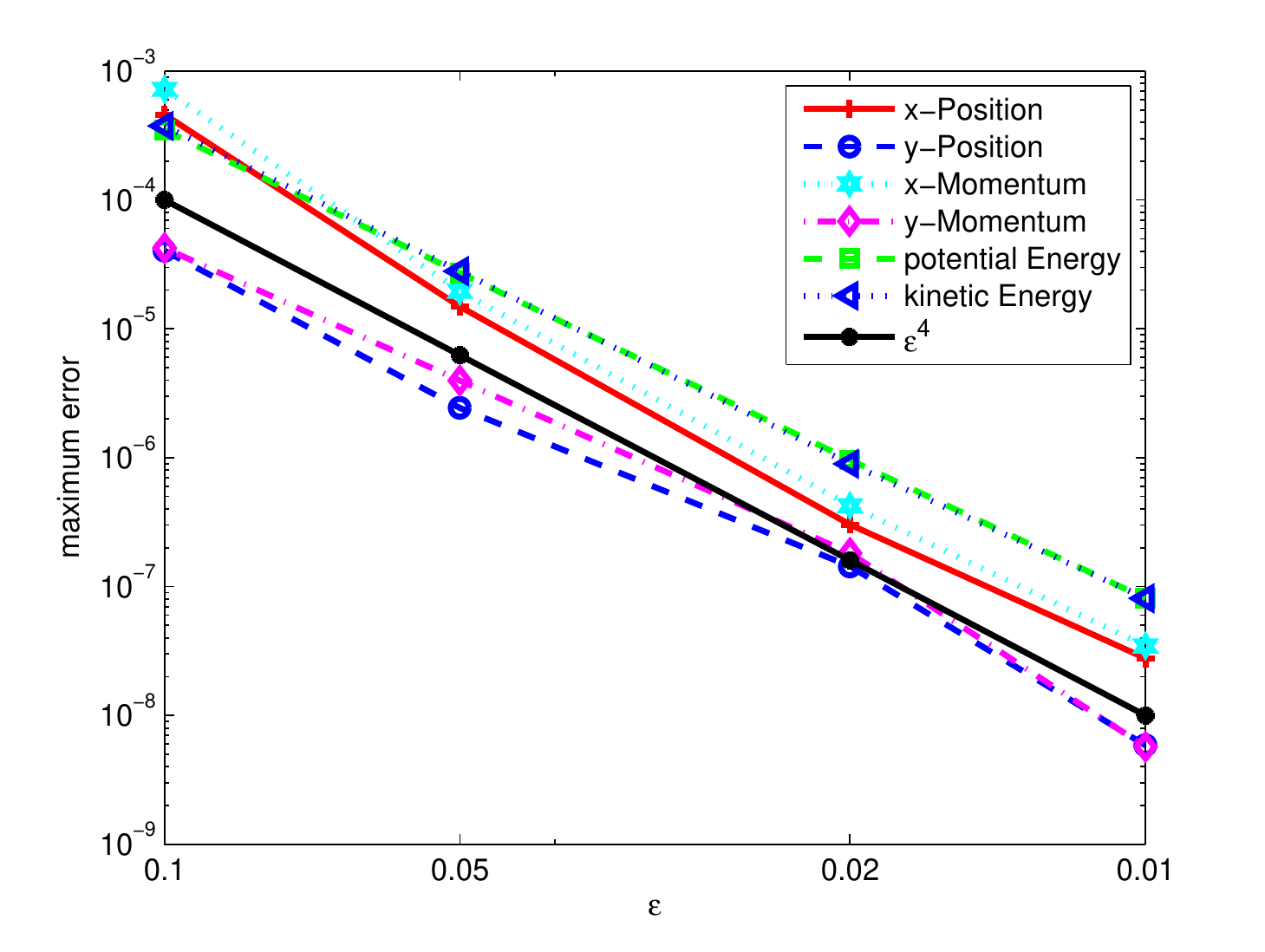}  
  } 
  \caption{The mean (a) and maximal (b) error over time of the expectation values of positions and momenta, kinetic and potential energy computed with the corrected algorithm (\ref{eq_approx_all}) as a function of $\eps$. The values for the number of sampling points and the time step sizes are given in Table~\ref{tab_diff_eps}. Both plots show an order of $\eps^4$.}
	\label{fig_diff_eps}
\end{figure}
\begin{table}[htbp]
	\label{tab_diff_eps}
	\caption{Numbers of sampling points $N_0$ and $N_2$, time step sizes $\tau_0$ and $\tau_2$ as well as the computing times $t^{\rm elaps}_0$ for the Egorov algorithm (\ref{eq_approx_a0}) and $t^{\rm elaps}_2$ for the corrections computed by the approximation scheme (\ref{eq_approx_a2}) that were used for the  results in Figure~\ref{fig_diff_eps}.  }
	\begin{indented}
	\item[]\begin{tabular}{ @{} l  l  l  l  l  l  l }
		\br 
		$\eps$ & $N_0$ & $\tau_0$ & $t^{\rm elaps}_0$ & $N_2$ & $\tau_2$ & $t^{\rm elaps}_2$ \\ \mr
		$0.1$ & $10^5$ & $0.1$ & $6.5$ sec & $500$ & $2^{-2}$ & $2.5$ sec\\ 
		$0.05$ & $10^6$ & $0.1$ & $46$ sec & $10^3$ & $2^{-3}$ & $6$ sec\\ 
		$0.02$ & $10^7$ & $0.1$ & $7.5$ min & $10^4$ & $2^{-5}$ & $2.5$ min \\ 
		$0.01$ & $10^8$ & $0.1$ &  $65$ min & $10^4$ & $2^{-5}$ & $2.5$ min \\ 
		\br
	\end{tabular}
	\end{indented}
\end{table}

\subsection{Discretization errors}
For the experiments shown in Figure~\ref{fig_err_N} and Figure~\ref{fig_err_tau} the values of the leading order approximation (\ref{eq_approx_a0}) are computed with $N_0=10^8$ sampling points and a time step size of $\tau_0=2^{-4}$.
In Figure~\ref{fig_err_N}, we examine the dependency of the absolute error of the expectation values from the number of particles $N_2$. To do so we compare expectation values coming from (\ref{eq_approx_all}) with those from our reference and compute the maximal as well as the mean error over the time interval $[0,15]$. For the computation of (\ref{eq_approx_a2}) we apply our fourth order splitting scheme with a time step size of $\tau_2=10^{-2}$. We can see that the mean as well as the maximal error decrease with an order slightly worse than $\eps^2 / N_2$ until they reach a lower bound of the order $\eps^4$. In Figure~\ref{fig_err_tau}, the mean and maximal error over time of the approximated expectation values depending on the step size $\tau_2$ is displayed. Here the number of sampling points in chosen as $N_2=10^4$. As expected, we observe that the mean as well as the maximum error decrease with an order of $\eps^2\tau_2^4$ until they reach values of the order $\eps^4$.  
\begin{figure}[htbp]
  \centering
  \subfigure[]{
    \includegraphics[width=6cm]{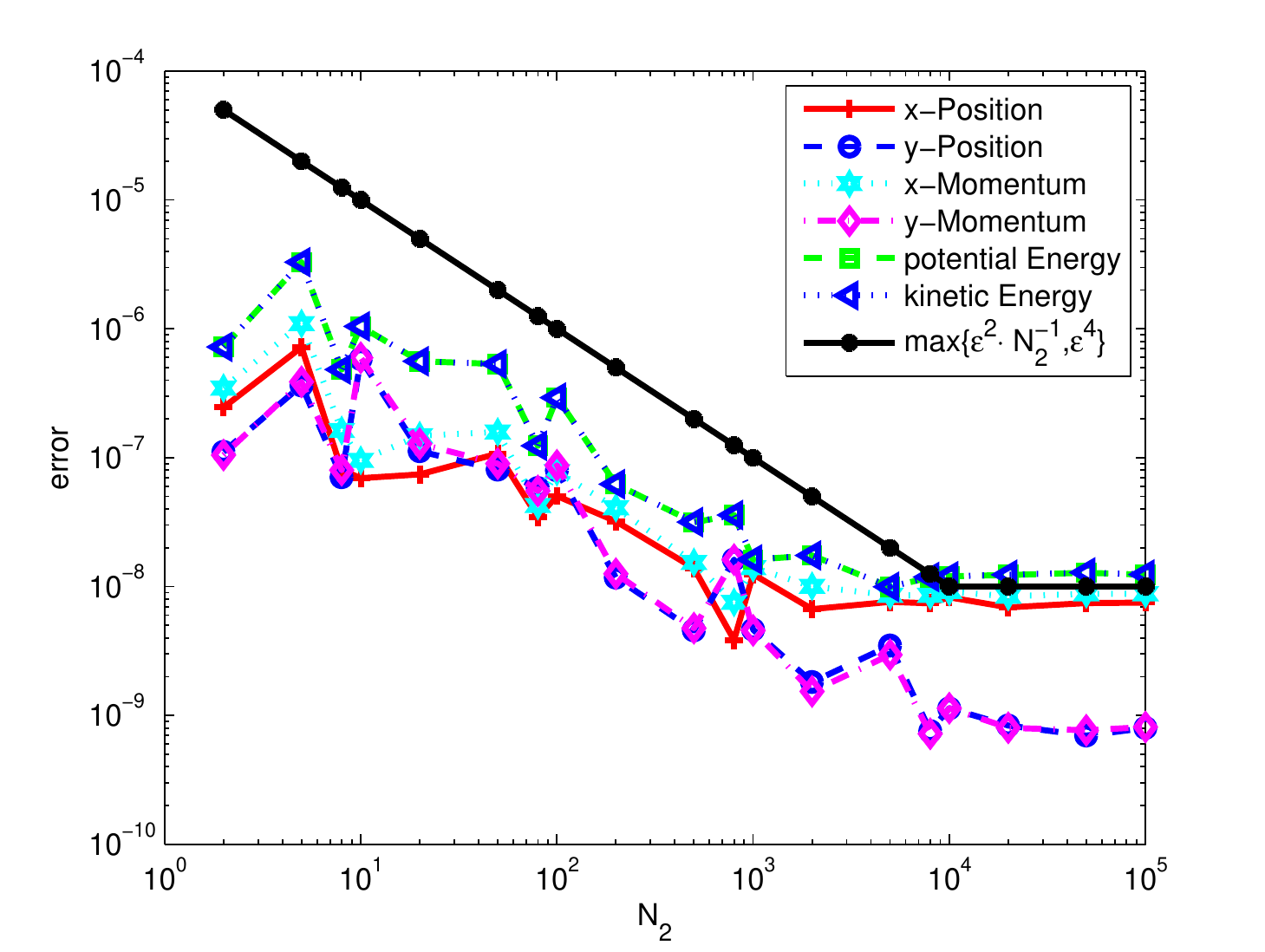} 
  }
  \subfigure[]{
    \includegraphics[width=6cm]{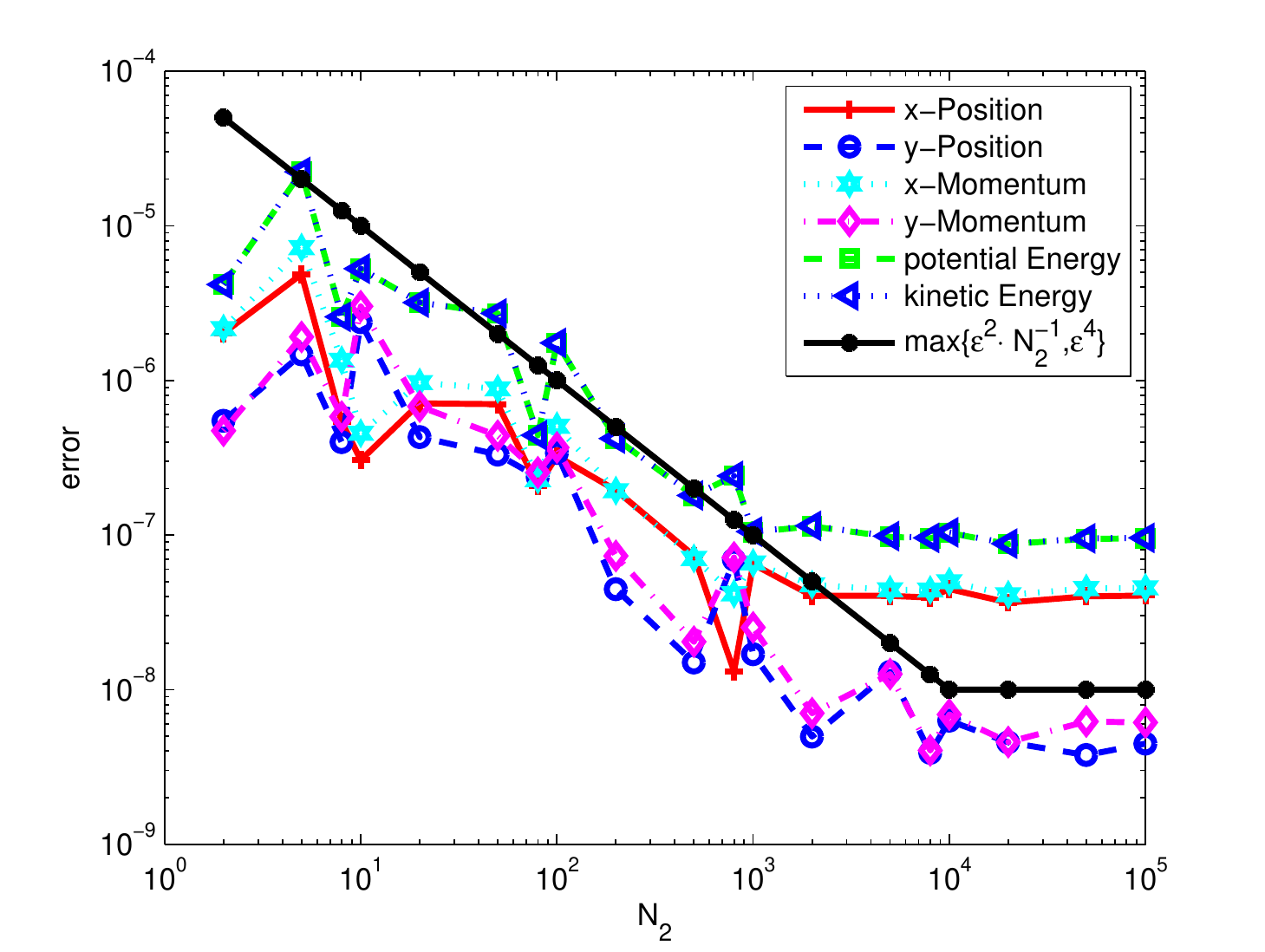}  
  } 
  \caption{The mean (a) and maximal (b) error over time of the expectation values of positions and momenta, kinetic and potential energy computed with the approximation (\ref{eq_approx_all}) as a function of $N_2$. The semiclassical parameter is chosen as $\eps=10^{-2}$ and the time step for the splitting scheme $F^{\tau_2}_4$ as $\tau_2=2^{-4}$. In both plots the order of the error is slightly worse than $\eps^2 / N_2$ and bounded from below by values of the order $\eps^4$. }
	\label{fig_err_N}
\end{figure}

\begin{figure}[htbp]
  \centering
  \subfigure[]{
    \includegraphics[width=6cm]{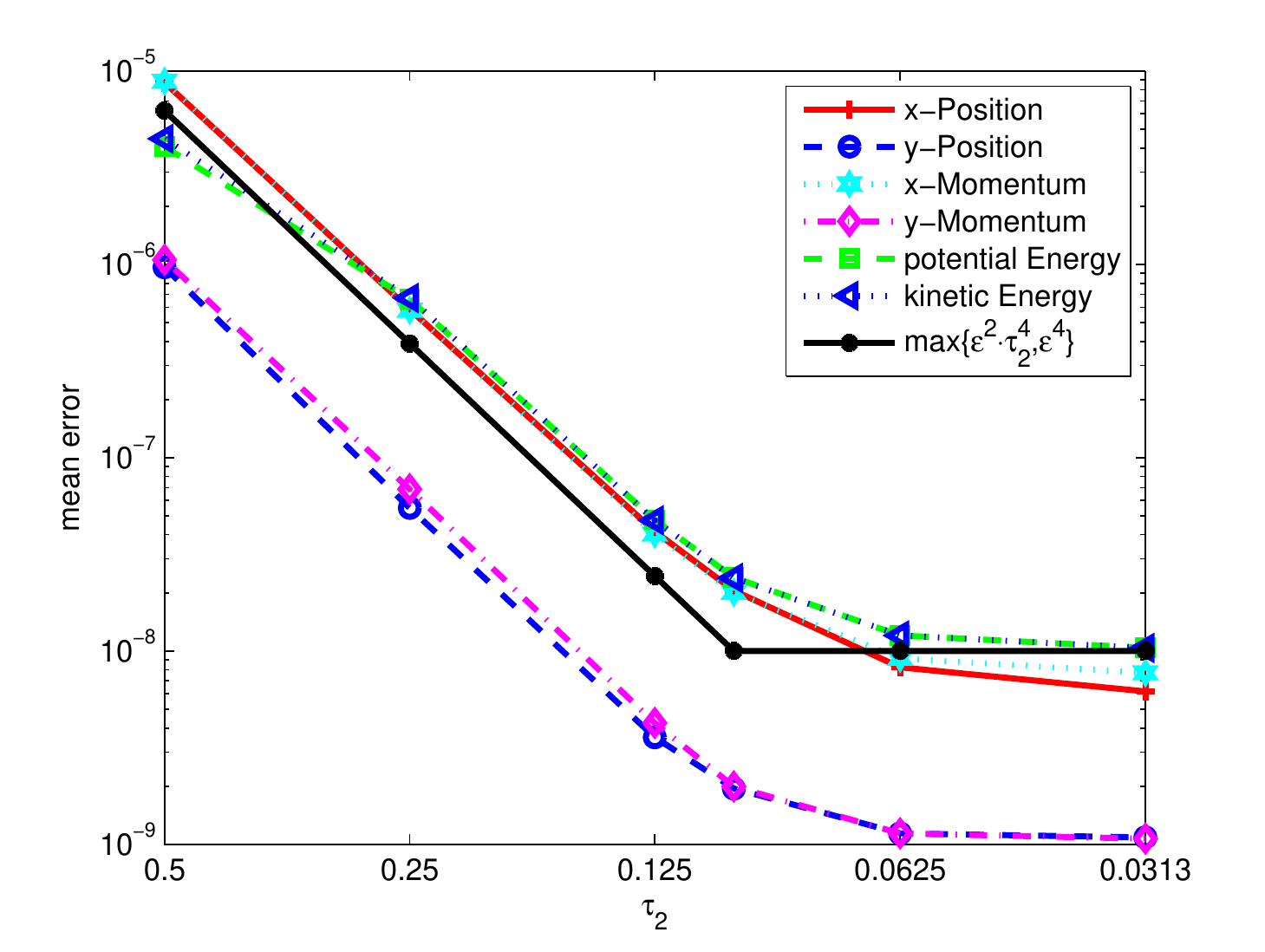} 
  }
  \subfigure[]{
    \includegraphics[width=6cm]{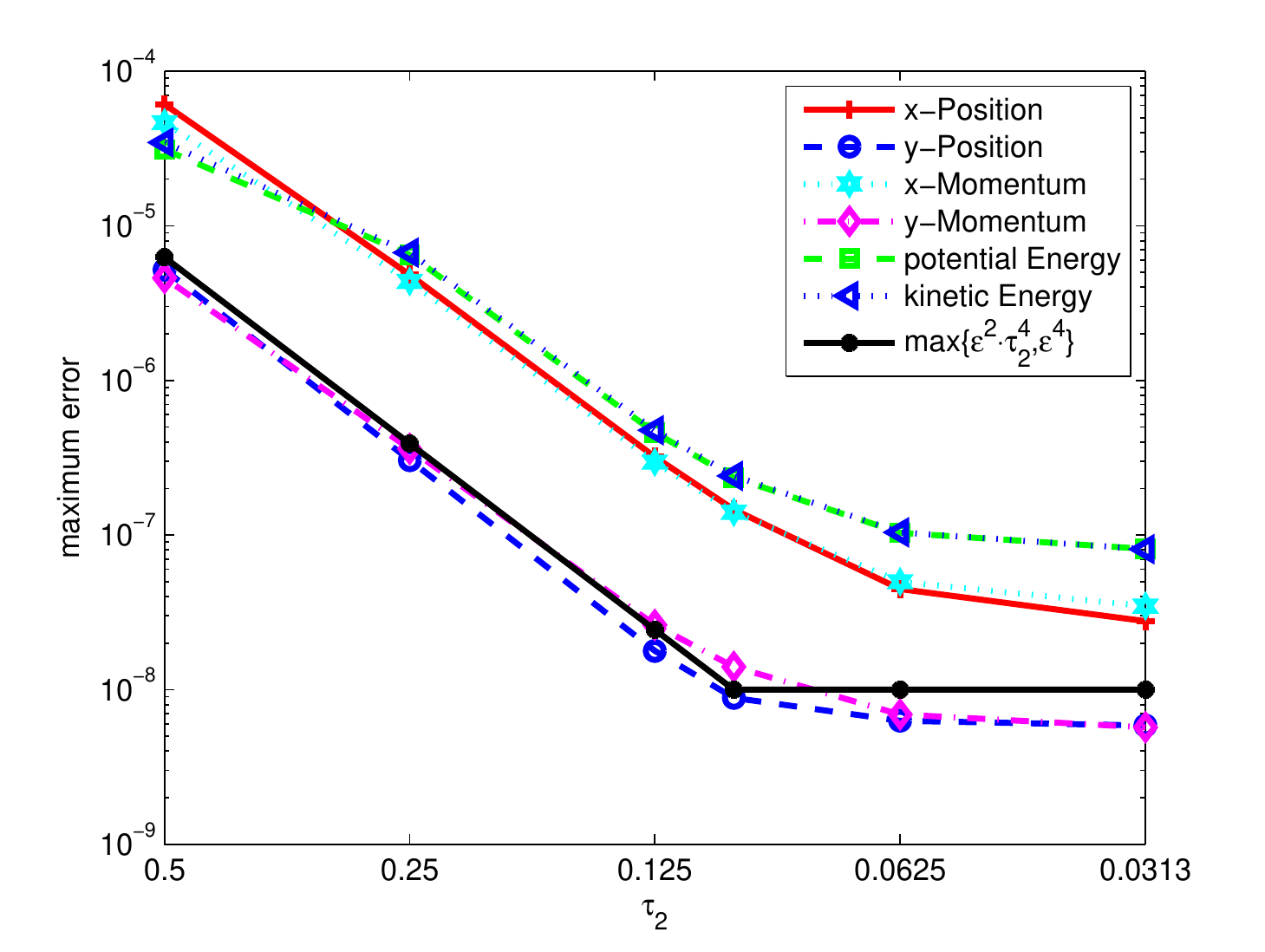}  
  } 
  \caption{The mean (a) and maximal (b) error over time of the expectation values of positions and momenta, kinetic and potential energy computed with the corrected algorithm (\ref{eq_approx_all}) as a function of $\tau_2$. The semiclassical parameter is chosen as $\eps=10^{-2}$ and the number of sampling points for (\ref{eq_approx_a2}) as $N_0=10^4$. In both plots the error is of the order $\eps^2 \tau_2^4$ but bounded from below by values of order $\eps^4$. }
	\label{fig_err_tau}
\end{figure}
\section*{Acknowledgement} We thank one of the anonymous referees for pointing us to the literature on Wigner trajectories. This research was supported by the German Research Foundation
(DFG), Collaborative Research Center SFB-TR 109.

\appendix
\section{Wigner functions and Weyl operators}\label{app:wigner}
Let $\eps>0$ be a small positive parameter. The $\eps$-scaled Wigner function $W_\psi:\R^{2d}\to\R$ of a square integrable function $\psi\in L^2(\R^d)$ is defined as
\begin{equation*}
W_\psi(q,p) = (2\pi\eps)^{-d} \int_{\R^d} \rme^{\rmi p\cdot y/\eps} \psi(q-\case12 y)\overline\psi(q+\case12 y) \rmd y,
\qquad (q,p)\in\R^{2d}, 
\end{equation*}
see for example \cite[\S1.8]{Fo89} or \cite[\S9]{Go11}. Using Plancherel's theorem it can easily be shown that $\|\psi\|=1$ implies
\begin{equation*}
\int_{\R^{2d}} W_\psi(z) \rmd z = 1.
\end{equation*}
The real-valued Wigner function $W_\psi$ can therefore be interpreted as a phase space description of the wave function $\psi$. In contrast to classical phase space distributions, the Wigner function may also attain negative values. The Wigner function is in close relation to the $\eps$-scaled Weyl quantization $\op(a)$ of Schwartz functions $a:\R^{2d}\to\R$, since
\begin{equation*}
\langle\psi,\op(a)\psi\rangle = \int_{\R^{2d}} a(z)W_\psi(z)\rmd z 
\end{equation*}
with
\begin{equation*}
(\op(a)\psi)(q) = (2\pi\eps)^{-d} \int_{\R^{2d}} a(\case12(q+y),p) \rme^{\rmi p\cdot(q-y)/\eps} \psi(y) \rmd y
\end{equation*}
for $\psi\in L^2(\R^d)$. Weyl quantization also extends to unbounded phase space functions $a$, if operator domains are defined suitably. For example, $-\eps^2\Delta = \op(|p|^2)$. The composition of Weyl operators is a Weyl operator
$\op(a)\op(b) = \op(c)$,  
and the symbol $c:\R^{2d}\to\R$ has an asymptoic expansion in powers of $\eps$, 
\begin{equation*}
c \sim \sum_{k\in\N} \left(\frac{\eps}{2i}\right)^k \{a,b\}_k,
\end{equation*}
where the $k$th generalized Poission bracket is defined as 
\begin{equation}
		\{a,b\}_k = 
			\sum_{|\alpha+\beta|=k} \frac{(-1)^{|\beta|}}{\alpha! \beta!}\, \partial_q^\alpha \partial_p^\beta b\, \partial_q^\beta \partial_p^\alpha a,
	\end{equation}
see \cite[Appendix]{BR02}. We note that the first generalized Poisson bracket coincides with the usual Poission bracket, $\{a,b\}_1 = \partial_p a \partial_q b - \partial_q a \partial_p b$. Moreover, the antisymmetry of the Possion bracket is also satisfied by its generalizations, if $k$ is odd.
	
\begin{lemma}
\label{prop_sym_poisson}
Let $a,b:\R^{2d}\to\R$ be smooth functions and $k\in\N$. Then, 
		\begin{eqnarray}
			\{a,b\}_k-\{b,a\}_k = \cases{ 0, & $k$ even \\ 2\{a,b\}_k, & $k$ odd	} 
		\end{eqnarray}
\end{lemma}
\begin{proof}
	By interchanging the multi-indices $\alpha$ and $\beta$ in the formula for $\{g,f\}_k$ we get
	\begin{eqnarray*}
		\{a,b\}_k-\{b,a\}_k = \sum_{|\alpha+\beta|=k} \frac{(-1)^{|\beta|}-(-1)^{|\alpha|}}{\alpha! \beta!} \partial_q^\alpha \partial_p^\beta b\, 
			\partial_q^\beta \partial_p^\alpha a.
	\end{eqnarray*} 
Since $|\alpha+\beta|=k$, we conclude that for even $k$, $|\alpha|$ is even if and only if $|\beta|$ is even 
	and for odd $k$, $|\alpha|$ is even if and only if $|\beta|$ is odd, which finishes the proof.
\end{proof}

In consequence, the commutator of two Weyl operators has an asymptotic expansion in odd powers of $\eps$,
\begin{equation*}
[\op(a),\op(b)] = \op(a)\op(b)-\op(b)\op(a) \sim 2 \sum_{k\in2\N+1} \left(\frac{\eps}{2i}\right)^k \op(\{a,b\}_k).
\end{equation*}



\section*{References}
\begin{thebibliography}{99}
\bibitem{BR02}
Bouzouina A and Robert D 2002 Uniform semiclassical estimates for the propagation of quantum observables
{\it Duke Math. J.} {\bf 111} 223--252
\bibitem{BH81}
Brown R and Heller E 1981 Classical trajectory approach to photodissociation: The Wigner method 
\JCP {\bf 75} 186--188
\bibitem{DM01}
Donoso A and Martens C 2001 Quantum tunneling using entangled classical trajectories
\PRL {\bf 87} 223202
\bibitem{DS99}
Dimassi M and Sj\"ostrand J 1999 {\it Spectral asymptotics in the semi-classical limit (LMS Lecture Note Series no. 268)} 
(Cambridge: Cambridge University Press)
\bibitem{Eg69}
Egorov Y 1969 On canonical transformations of pseudodifferential operators 
{\it Uspekhi Mat. Nauk} {\bf 24} 235--236
\bibitem{Fo89}
Folland G 1989 {\it Harmonic Analysis in Phase Space} (Princeton: Princeton University Press)
\bibitem{Go11}
de Gosson M 2011 {\it Symplectic Methods in Harmonic Analysis and in Mathematical Physics} (Basel: Birkh\"auser)
\bibitem{He76}
Heller E 1976 Wigner phase space method: Analysis for semiclassical applications
\JCP {\bf 65} 1289--1298
\bibitem{KC05}
Kryvohuz M and Cao J 2005 Quantum-classical correspondence in response theory
\PRL {\bf 95} 180405
\bibitem{LR10}
Lasser C and R\"oblitz S 2010 Computing expectation values for molecular quantum dynamics 
{\it SIAM J. Sci. Comput.} {\bf 32} 1465--1483
\bibitem{L92}
Lee H 1992 Wigner trajectories of a Gaussian wave packet perturbed by a weak potential
{\it Found. Phys.} {\bf 22} 995--1010
\bibitem{L95}
Lee H 1995 Theory and application of the quantum phase-space distrubtion functions
{\it Phys. Rep.} {\bf 259} 147--211
\bibitem{LS80}
Lee H and Scully M 1980 A new approach to molecular collisions: Statistical quasiclassical method
\JCP {\bf 73} 2238--2242
\bibitem{LS82}
Lee H and Scully M 1982 Wigner phase-space description of a Morse oscillator
\JCP {\bf 77} 4604--4610
\bibitem{LM07}
Liu J and Miller W 2007 Real time correlation function in a single phase space integral
beyond the linearized semiclassical initial value representation
\JCP {\bf 126} 234110
\bibitem{Ma02}
Martinez A 2002 {\it An Introduction to Semiclassical and Microlocal Analysis} ({\it Universitext}) (New York: Springer)
\bibitem{Mi74}
Miller W 1974 Quantum mechanical transition state theory and a new semiclassical model
for reaction rate constants \JCP {\bf 61} 1823--1834
\bibitem{Pu06}
Pulvirenti M 2006 Semiclassical expansion of Wigner functions
\JMP {\bf 47} 052103
\bibitem{Ro87}
Robert D 1987 {\it Autour de l'Approximation Semi-Classique} ({\it Progress in Mathematics} vol 68) (Boston: Birkh\"auser)
\bibitem{ST01}
Spohn H and Teufel S 2001 Adiabatic decoupling and time-dependent Born-Oppenheimer theory
{\it Commun. Math. Phys.} {\bf 224} 113--172 
\bibitem{Ta91}
Taylor M 1991 {\it Pseudodifferential Operators and Nonlinear PDE} ({\it Progress in Mathematics} vol 100) (Boston: Birkh\"auser)
\bibitem{TW04}
Thoss M and Wang H 2004 Semiclassical description of molecular dynamics based on initial-value representation methods
{\it Annu. Rev. Phys. Chem.} {\bf 55} 299--332
\bibitem{WSW08}
Waalkens H, Schubert R and Wiggins S 2008 
Wigner's dynamical transition state theory in phase space: classical and quantum
\NL {\bf 21} R1--R118
\bibitem{WSM98}
Wang H, Sun X and Miller H 1998 Semiclassical approximations for the calculation of thermal rate constants
for chemical reactions in complex molecular systems \JCP {\bf 108} 9726--9736
\bibitem{Wi32}
Wigner E 1932 On the quantum correction for thermodynamic equilibrium 
\PR {\bf 40} 749--759
\bibitem{Zw12}
Zworski M 2012 {\it Semiclassical Analysis} ({\it Graduate Studies in Mathematics} vol 138) (Providence: AMS)
\bibitem{Yo90}
Yoshida H 1990 {\it Construction of higher order symplectic integrators} {\it Physics Letters A} {\bf 150} no.5,6,7 262--268
\bibitem{HWL06}
Hairer E and Lubich C and Wanner G 2006 {\it Geometric Numerical Integration. Structure-Preserving Algorithms for Ordinary Differential Equations, Second Edition} 
({\it Springer Series in Computational Mathematics} vol 31) (Berlin: Springer-Verlag)
\end{thebibliography}
\end{document}